\documentclass[reqno,openright,11pt,twoside]{amsart}

\usepackage{etoolbox}

\usepackage{datetime}

\usepackage{bm, xspace}

\usepackage[pagebackref=true, colorlinks=true, citecolor=blue]{hyperref}

\usepackage[top=1.4in, bottom=1.4in, left=1.75in, right=1.75in]{geometry}

\usepackage{cleveref}

\newbool{IncludeNavierBCSection}

\newbool{arXivFormat}
\booltrue{arXivFormat}
    
\newcommand{\IfarXivElse}[2]{
    \ifbool{arXivFormat}
        {#1}{#2}
    }

\newtheorem{theorem}{Theorem}[section]
\newtheorem{lemma}[theorem]{Lemma}
\newtheorem{prop}[theorem]{Proposition}
\newtheorem{cor}[theorem]{Corollary}
\newtheorem{conj}[theorem]{Conjecture}

\theoremstyle{definition}

\newtheorem{remark}[theorem]{Remark}
\newtheorem*{remark*}{Remark}     

\numberwithin{equation}{section}

\newcommand{\abs}[1]{\left\vert#1\right\vert}

\newcommand{\innp}[1]{\ensuremath{\left< #1 \right>}}

\newcommand{\BoldTau}
    {\mbox{\boldmath $\tau$}}

\newcommand{\BB}[1]{\ensuremath{\mathbb{#1}}}
\newcommand{\R}{\ensuremath{\BB{R}}} %
\newcommand{\iny}{\ensuremath{\infty}}
\newcommand{\grad}{\ensuremath{\nabla}}
\DeclareMathOperator{\dv}{div} %
\DeclareMathOperator{\curl}{curl} %
\DeclareMathOperator{\dist}{dist} %
\DeclareMathOperator{\supp}{supp} %
\DeclareMathOperator{\weak}{weak} %
\newcommand{\wh}{\ensuremath{\widehat}}
\newcommand{\prt}{\ensuremath{\partial}}
\newcommand{\brac}[1]{\ensuremath{\left[ #1 \right]}}
\newcommand{\pr}[1]{\ensuremath{\left( #1 \right) }}
\newcommand{\set}[1]{\ensuremath{\left\{ #1 \right\}}}

\newcommand{\norm}[1]{\ensuremath{\left\Vert #1 \right\Vert}}
\newcommand{\smallnorm}[1]{\ensuremath{\Vert #1 \Vert}}
\newcommand{\refS}[1]{Section~\ref{S:#1}}
\newcommand{\refT}[1]{Theorem~\ref{T:#1}}
\newcommand{\refTAnd}[2]{Theorems~\ref{T:#1} and \ref{T:#2}}
\newcommand{\refL}[1]{Lemma~\ref{L:#1}}

\newcommand{\refC}[1]{Corollary~\ref{C:#1}}
\newcommand{\refE}[1]{Equation~(\ref{e:#1})}

\newcommand{\refR}[1]{Remark~\ref{R:#1}}
\newcommand{\eps}{\ensuremath{\epsilon}}
\newcommand{\Cal}[1]{\ensuremath{\mathcal{#1}}}
\newcommand{\al}{\ensuremath{\alpha}}

\newcommand{\pdx}[2]{\frac{\prt #1}{\prt #2}}

\newcommand{\diff}[2]{\frac{ d#1}{d#2}}

\newcommand{\ol}{\overline}

\newcommand{\n}{\bm{n}}

\newcommand{\ToDo}[1]{
		{\textbf{[#1]}}}

\newcommand{\Example}[1]{\bigskip\noindent\textbf{Example #1}: }

\newcommand{\ReturnExample}[1]{\bigskip\noindent\textbf{Return to Example #1}: }

\renewcommand{\matrix}[2]{\begin{pmatrix} #1 \\ #2 \end{pmatrix}}

\newcommand{\tmatrix}[3]{\begin{pmatrix} #1 \\ #2 \\ #3 \end{pmatrix}}

\crefname{cor}{Corollary}{Corollaries} 
									   
\crefname{lemma}{Lemma}{Lemmas}	       

\crefname{section}{Section}{Sections}
\Crefname{section}{Section}{Sections}

\crefname{appendix}{Appendix}{Appendices}
\Crefname{appendix}{Appendix}{Appendices}

\crefname{theorem}{Theorem}{Theorems}
\Crefname{theorem}{Theorem}{Theorems}

\crefname{prop}{Proposition}{Propositions}
\Crefname{prop}{Proposition}{Propositions}

\crefname{conj}{Conjecture}{Conjectures}
\Crefname{conj}{Conjecture}{Conjectures}

\crefname{definition}{Definition}{Definitions}
\Crefname{definition}{Definition}{Definitions}

\crefname{remark}{Remark}{Remarks}
\Crefname{remark}{Remark}{Remarks}

\crefname{assumption}{Assumption}{Assumptions}
\Crefname{assumption}{Assumption}{Assumptions}

\crefname{conj}{Conjecture}{Conjectures}
\Crefname{conj}{Conjecture}{Conjectures}

\crefformat{equation}{(#2#1#3)}
\crefrangeformat{equation}{(#3#1#4) through (#5#2#6)}
\crefmultiformat{equation}
    {(#2#1#3)}%
    { and~(#2#1#3)}
    {, (#2#1#3)}
    { and~(#2#1#3)}

\newcommand{\Part}[1]
{
\bigskip
\begin{center}
\textbf{#1}
\phantomsection   
\addcontentsline{toc}{chapter}{#1}
\end{center}
}

\numberwithin{equation}{section}

\newcommand{\Ignore}[1]{}

\newcommand{\Holder}
    {H\"{o}lder }

\newcommand{\Holders}
    {H\"{o}lder's }

\renewcommand{\matrix}[2]{\begin{pmatrix} #1 \\ #2 \end{pmatrix}}

\begin{document}

\newdateformat{mydate}{\THEDAY~\monthname~\THEYEAR}

\title
    [Vanishing viscosity: observations]
    {Observations on the vanishing viscosity limit}

\author{James P. Kelliher}
\address{Department of Mathematics, University of California, Riverside, 900 University Ave.,
Riverside, CA 92521}
\curraddr{Department of Mathematics, University of California, Riverside, 900 University Ave.,
Riverside, CA 92521}
\email{kelliher@math.ucr.edu}




\begin{abstract}
Whether, in the presence of a boundary, solutions of the Navier-Stokes equations converge to a solution to the Euler equations in the vanishing viscosity limit is unknown. In a seminal 1983 paper, Tosio Kato showed that the vanishing viscosity limit is equivalent to having sufficient control of the gradient of the Navier-Stokes velocity in a boundary layer of width proportional to the viscosity. In a 2008 paper, the present author showed that the vanishing viscosity limit is equivalent to the formation of a vortex sheet on the boundary. We present here several observations that follow on from these two papers.
\Ignore{ 
We make several observations regarding the vanishing viscosity limit, primarily regarding the control of the total mass of vorticity and the conditions in Tosio Kato's seminal 1983 paper \cite{Kato1983} shown by him to be equivalent to the vanishing viscosity limit.
} 
\end{abstract}

\date{(compiled on {\dayofweekname{\day}{\month}{\year} \mydate\today)}}

\maketitle

\vspace{-1.5em}

\begin{small}
    \begin{flushright}
        Compiled on \textit{\textbf{\dayofweekname{\day}{\month}{\year} \mydate\today}}
    \end{flushright}
\end{small}

\vspace{-1.5em}

\renewcommand\contentsname{}   
\begin{small}
    \tableofcontents
\end{small}


\noindent 
The Navier-Stokes equations for a viscous incompressible fluid in a domain $\Omega \subseteq \R^d$, $d \ge 2$, with no-slip boundary conditions can be written,
\begin{align*}
    (NS)
    \left\{
        \begin{array}{rl}
            \prt_t u + u \cdot \grad u + \grad p = \nu \Delta u + f
                &\text{ in } \Omega, \\
            \dv u = 0
                &\text{ in } \Omega, \\
            u = 0
                &\text{ on } \Gamma := \prt \Omega.
        \end{array}  
    \right.
\end{align*}
The Euler equations modeling inviscid incompressible flow on such a domain with no-penetration boundary conditions can be written,
\begin{align*}
    (EE)
    \left\{
        \begin{array}{rl}
            \prt_t \ol{u} + \ol{u} \cdot \grad \ol{u} + \grad p = \nu \Delta \ol{u} + \ol{f}
                &\text{ in } \Omega, \\
            \dv \ol{u} = 0
                &\text{ in } \Omega, \\
            \ol{u} \cdot \n = 0
                &\text{ on } \Gamma.
        \end{array}  
    \right.
\end{align*}
Here, $u = u_\nu$ and $\ol{u}$ are velocity fields, while $p$ and $\ol{p}$ are pressure (scalar) fields. The external forces, $f$, $\ol{f}$, are vector fields. (We adopt here the notation of Kato in \cite{Kato1983}.)

We assume throughout that $\Omega$ is bounded and $\Gamma$ has $C^2$ regularity, and write $\n$ for the outward unit normal vector.

The limit,
\begin{align*}
    (VV) \qquad
        u \to \ol{u} \text{ in } L^\iny(0, T; L^2(\Omega)),
\end{align*}
we refer to as the \textit{classical vanishing viscosity limit}. Whether it holds in general, or fails in any one instance, is a major open problem in mathematical fluids mechanics.

In \cite{K2006Kato, K2008VVV} a number of conditions on the solution $u$ were shown to be equivalent to ($VV$). 
The focus in \cite{K2006Kato} was on the size of the vorticity or velocity in a layer near the boundary, while the focus in \cite{K2008VVV} was on the accumulation of vorticity on the boundary. The work we present here is in many ways a follow-on to \cite{K2006Kato, K2008VVV}, each of which, especially \cite{K2006Kato}, was itself an outgrowth of Tosio Kato's seminal paper \cite{Kato1983} on the vanishing viscosity limit, ($VV$). 

This paper is divided into two themes. The first theme concerns the accumulation of vorticity---on the boundary, in a boundary layer, or in the bulk of the fluid. It explores the consequences of having control of the total mass of vorticity or, more strongly, the $L^1$-norm of the vorticity for solutions to ($NS$).

We re-express in a specifically 3D form the condition for vorticity accumulation on the boundary from \cite{K2008VVV} in \cref{S:3DVersion}.
In \cref{S:LpNormsBlowUp}, we show that if ($VV$) holds then the $L^p$ norms of the vorticity for solutions to ($NS$) must blow up for all $p > 1$ as $\nu \to 0$ except in very special circumstances. This leaves only the possibility of control of the vorticity's $L^1$ norm. Assuming such control, we show in \cref{S:ImprovedConvergence} that when ($VV$) holds we can characterize the accumulation of vorticity on the boundary more strongly than in \cite{K2008VVV}.

In \cref{S:BoundaryLayerWidth}, we show that if we measure the width of the boundary layer by the size of the $L^1$-norm of the vorticity then the layer has to be wider than that of Kato if ($VV$) holds. We push this analysis further in \cref{S:OptimalConvergenceRate} to obtain the theoretically optimal convergence rate when the initial vorticity has nonzero total mass, as is generic for non-compatible initial data. We turn a related observation into a conjecture concerning the connection between the vanishing viscosity limit and the applicability of the Prandtl theory.

In \cref{S:SomeConvergence}, we show that the arguments in \cite{K2008VVV} lead to the conclusion that some kind of convergence of a subsequence of the solutions to ($NS$) always occurs in the limit as $\nu \to 0$, but not necessarily to a solution to the Euler equations.
 
The second theme more directly addresses Tosio Kato's conditions from  \cite{Kato1983} that are equivalent to ($VV$). We also deal with the closely related condition from \cite{K2006Kato} that uses vorticity in place of the gradient of the velocity that appears in one of Kato's conditions.

We derive in \cref{S:EquivCondition} a condition on the solution to ($NS$) on the boundary that is equivalent in 2D to ($VV$), giving a number of examples to which this condition applies in \cref{S:Examples}.

In \cref{S:BardosTiti} we discuss some interesting recent results of Bardos and Titi that they developed using dissipative solutions to the Euler Equations. We show how weaker, though still useful, 2D versions of these results can be obtained using direct elementary methods. 

We start, however, in \cref{S:Background} with the notation and definitions we will need, and a summary of the pertinent results of \cite{K2006Kato, K2008VVV, Kato1983}.

%
%
\section{Definitions and past results}\label{S:Background}

\noindent
We define the classical function spaces of incompressible fluids,
\begin{align*}
    H &= \set{u \in (L^2(\Omega))^d: \dv u = 0 \text{ in } \Omega, \,
                u \cdot \mathbf{n} = 0 \text{ on } \Gamma}
\end{align*}
with the $L^2$-norm and
\begin{align*}                
    V &= \set{u \in (H_0^1(\Omega))^d: \dv u = 0 \text{ in } \Omega}
\end{align*}
with the $H^1$-norm. We denote the $L^2$ or $H$ inner product by $(\cdot, \cdot)$. If $v$, $w$ are vector fields then $(v, w) = (v^i, w^i)$, where we use here and below the common convention of summing over repeated indices. Similarly, if $M$, $N$ are matrices of the same dimensions then $M \cdot N = M^{ij} N^{ij}$ and
\begin{align*}
   (M, N)
       = (M^{ij}, N^{ij})
       = \int_\Omega M \cdot N.
\end{align*}

We will assume that $u$ and $\ol{u}$ satisfy the same initial conditions,
\begin{align*}
    u(0) = u_0, \quad \ol{u}(0) = u_0,
\end{align*}
and that $u_0$ is in $C^{k + \eps}(\Omega) \cap H$, $\eps > 0$, where $k =
1$ for two dimensions and $k = 2$ for 3 and higher dimensions, and that
$f = \ol{f} \in C^1_{loc}(\R; C^1(\Omega))$. Then as shown in
\cite{Koch2002} (Theorem 1 and the remarks on p. 508-509), there is some $T
> 0$ for which there exists a unique solution,
\begin{align}\label{e:ubarSmoothness}
    \ol{u}
        \text{ in } C^1([0, T]; C^{k + \eps}(\Omega)),
\end{align}
to ($EE$). In two dimensions, $T$ can be arbitrarily large, though it is only known that
some positive $T$ exists in three and higher dimensions.

With such initial velocities, we are assured that there are weak solutions to $(NS)$, unique in 2D. Uniqueness of these weak solutions is not known in three and higher dimensions, so by $u = u_\nu$ we mean any of these solutions chosen arbitrarily. We never employ strong or classical solutions to $(NS)$.

\Ignore{ 
It follows, assuming that $f$ is in $L^1([0, T]; L^2(\Omega))$, that for such solutions,
\begin{align}\label{e:NSVariationalIdentity}
    \begin{split}
        &(u(t), \phi(t)) - (u(0), \phi(0)) \\
            &\qquad= \int_0^t \brac{(u, u \cdot \grad \phi)
                - \nu (\grad u, \grad
                    \phi) + (f, \phi) + (u, \prt_t \phi))} \, dt
   \end{split}
\end{align}
for all $\phi$ in $C^1([0, T] \times \Omega) \cap C^1([0, T]; V)$. 
} 

We define $\gamma_\mathbf{n}$ to be the boundary trace operator for the normal component of a vector field in $H$ and write
\begin{align}\label{e:RadonMeasures}
    \Cal{M}(\ol{\Omega}) \text{ for the space of Radon measures on } \ol{\Omega}.
\end{align}
That is, $\Cal{M}(\ol{\Omega})$ is the dual space of $C(\ol{\Omega})$. We let $\mu$ in $\Cal{M}(\ol{\Omega})$ be the measure supported on $\Gamma$ for which  $\mu\vert_\Gamma$ corresponds to Lebesgue measure on $\Gamma$ (arc length for $d = 2$, area for $d = 3$). Then $\mu$ is also a member of $H^1(\Omega)^*$, the dual space of $H^1(\Omega)$.

We define the vorticity $\omega(u)$ to be the $d \times d$ antisymmetric matrix,
\begin{align}\label{e:VorticityRd}
    \omega(u) = \frac{1}{2}\brac{\grad u - (\grad u)^T},
\end{align}
where $\grad u$ is the Jacobian matrix for $u$: $(\grad u)^{ij} = \prt_j u^i$.
When working specifically in two dimensions, we alternately define the vorticity as the scalar curl of $u$: 
\begin{align}\label{e:VorticityR2}
	\omega(u) = \prt_1 u^2 - \prt_2 u^1.
\end{align}

Letting $\omega = \omega(u)$ and $\ol{\omega} = \omega(\ol{u})$, we define the following conditions:
\begingroup
\allowdisplaybreaks
\begin{align*}
	(A) & \qquad u \to \ol{u} \text{ weakly in } H
					\text{ uniformly on } [0, T], \\
	(A') & \qquad u \to \ol{u} \text{ weakly in } (L^2(\Omega))^d
					\text{ uniformly on } [0, T], \\
	(B) & \qquad u \to \ol{u} \text{ in } L^\iny([0, T]; H), \\
	(C) & \qquad \grad u \to \grad \ol{u} - \innp{\gamma_\mathbf{n} \cdot, \ol{u} \mu} 
				\text{ in } ((H^1(\Omega))^{d \times d})^*
					   \text{ uniformly on } [0, T], \\
	(D) & \qquad \grad u \to \grad \ol{u} \text{ in } (H^{-1}(\Omega))^{d \times d}
					   \text{ uniformly on } [0, T], \\
	(E) & \qquad \omega \to \ol{\omega}
					- \frac{1}{2} \innp{\gamma_\mathbf{n} (\cdot - \cdot^T),
								\ol{u} \mu}
					\text{ in } 
	 				((H^1(\Omega))^{d \times d})^*
	 				   \text{ uniformly on } [0, T], \\
	(F) & \qquad \omega \to \ol{\omega}
					\text{ in } 
	 				 (H^{-1}(\Omega))^{d \times d}
					   \text{ uniformly on } [0, T].
\end{align*}
\endgroup
We stress that $(H^1(\Omega))^*$ is the dual space of $H^1(\Omega)$, in contrast to $H^{-1}(\Omega)$, which is the dual space of $H^1_0(\Omega)$.

The condition in $(B)$ is the classical vanishing viscosity limit of ($VV$).

We will make the most use of condition $(E)$, which more explicitly means that
\begin{align}\label{e:EExplicit}
	(\omega(t), M)
		\to (\ol{\omega}(t), M) - \frac{1}{2}\int_{\Gamma}
			((M - M^T) \cdot \mathbf{n}) \cdot \ol{u}(t)
				\text{ in } L^\iny([0, T])
\end{align}
for any $M$ in $(H^1(\Omega))^{d \times d}$.

In two dimensions, defining the vorticity as in \refE{VorticityR2}, we also define the following two conditions:
\begin{align*}
	(E_2) & \qquad \omega \to \ol{\omega} - (\ol{u} \cdot \BoldTau) \mu 
				\text{ in } (H^1(\Omega))^*
					   \text{ uniformly on } [0, T], \\
	(F_2) & \qquad \omega \to \ol{\omega} \text{ in } H^{-1}(\Omega)
					   \text{ uniformly on } [0, T].
\end{align*}
Here, $\BoldTau$ is the unit tangent vector on $\Gamma$ that is obtained by rotating the outward unit normal vector $\mathbf{n}$ counterclockwise by $90$ degrees.

\Ignore{ 
Condition ($E_2$) means that
\begin{align*}
	(\omega(t), f)
		\to (\ol{\omega}(t), f) - \int_{\Gamma} (\ol{u}(t) \cdot \BoldTau) f
			\text{ in } L^\iny([0, T])
\end{align*}
for any $f$ in $H^1(\Omega)$.
} 

\refT{VVEquiv} is proved in \cite{K2008VVV} ($(A) \implies (B)$ having been proved in \cite{Kato1983}), to which we refer the reader for more details.

\begin{theorem}[\cite{K2008VVV}]\label{T:VVEquiv}
	Conditions ($A$), ($A'$), ($B$), ($C$), ($D$), and ($E$) are equivalent
	(and each implies condition ($F$)).
	In two dimensions, condition ($E_2$) and, when $\Omega$ is simply connected, ($F_2$)
	are equivalent to the other conditions.\footnote{The restriction that $\Omega$ be
	simply connected for the equivalence of ($F_2$) was not, but should
	have been, in the published version of \cite{K2008VVV}.}
\end{theorem}

\cref{T:VVEquiv} remains silent about rates of convergence, but examining the proof of it in \cite{K2008VVV} easily yields the following:
\begin{theorem}\label{T:ROC}
    Assume that ($VV$) holds with
    \begin{align*}
        \norm{u - \ol{u}}_{L^\iny(0, T; L^2(\Omega))}
            \le F(\nu)
    \end{align*}
    for some fixed $T > 0$. Then
    \begin{align*}
        \norm{(u(t) - \ol{u}(t), v)}_{L^\iny([0, T])}
            \le F(\nu) \norm{v}_{L^2(\Omega)}
                \text{ for all } v \in (L^2(\Omega))^d
    \end{align*}
    and
    \begin{align*}
        \norm{(\omega(t) - \ol{\omega}(t), \varphi)}_{L^\iny([0, T])}
            \le F(\nu) \norm{\grad \varphi}_{L^2}
                \text{ for all } \varphi \in H_0^1(\Omega).
    \end{align*}
\end{theorem}

\begin{remark}\label{R:ROCOthers}
    \cref{T:ROC} gives the rates of convergence for ($A$) and ($F_2$);
    the rates for ($C$), ($D$), ($E$), and ($E_2$) are like those given for ($F_2$)
    (though the test function, $\varphi$, will lie in different spaces).
\end{remark}

In \cite{Kato1983}, Tosio Kato showed that ($VV$) is equivalent to
\begin{align*}
    \nu \int_0^T \norm{\grad u(s)}_{L^2(\Omega)}^2 \, dt \to 0
        \text{ as } \nu \to 0
\end{align*}
and to
\begin{align}\label{e:KatoCondition}
    \nu \int_0^T \norm{\grad u(s)}_{L^2(\Gamma_{c \nu})}^2 \, dt \to 0
        \text{ as } \nu \to 0.  
\end{align}
Here, and in what follows, $\Gamma_\delta$ is a boundary layer in $\Omega$ of width $\delta > 0$.

In \cite{K2006Kato} it is shown that in \cref{e:KatoCondition}, the gradient can be replaced by the vorticity, so ($VV$) is equivalent to
\begin{align}\label{e:KellCondition}
    \nu \int_0^T \norm{\omega(s)}_{L^2(\Gamma_{c \nu})}^2 \, dt \to 0
        \text{ as } \nu \to 0.  
\end{align}
Note that the necessity of \cref{e:KellCondition} follows immediately from \cref{e:KatoCondition}, but the sufficiency does not, since on the inner boundary of $\Gamma_{c \nu}$ there is no boundary condition of any kind.

We also mention the works \cite{TW1998, W2001}, which together establish conditions equivalent to \refE{KatoCondition}, with a boundary layer slightly larger than that of Kato, yet only involving the tangential derivatives of either the normal or tangential components of $u$ rather than the full gradient. These conditions will not be used in the present work, however.

\Ignore{ 
The setup and notation are that of \cite{K2008VVV, K2006Kato}, and is largely inherited from \cite{Kato1983}: Weak solutions to the Navier-Stokes equations in a bounded domain, $\Omega$, having $C^2$-boundary, $\Gamma$, are denoted by $u$, the viscosity, $\nu > 0$, being implied by context. Weak (or often strong) solutions to the Euler equations are denoted by $\ol{u}$. Except in \refS{NavierBCs}, we use homogeneous Dirichlet conditions ($u = 0$) for the Navier-Stokes equations and we in any case always use no-penetration conditions ($u \cdot \n = 0$) for the Euler equations. Here, $\n$ is the outward normal to the boundary. We use $\omega = \omega(u)$ to be the curl of $u$, defined to be $\prt_1 u^2 - \prt_2 u^1$ in $2D$ and the antisymmetric part of $\grad u$ in higher dimensions. Similarly for $\ol{\omega} = \omega(\ol{u})$.

We denote the $L^2$-inner product by $(\cdot, \cdot)$, and write  $V$ for the space of all divergence-free vector fields in $H_0^1(\Omega)$. We will also use the related function space $H$ of divergence-free vector fields $v$ in $L^2(\Omega)$ with $v \cdot \mathbf{n} = 0$ on $\Gamma$ in the sense of a trace.

See \cite{K2008VVV, K2006Kato} for more details.
} 

%
%
\Part{Theme I: Accumulation of vorticity}

%
%
\section{A 3D version of vorticity accumulation on the boundary}\label{S:3DVersion}

\noindent
In \cref{T:VVEquiv}, the vorticity is defined to be the antisymmetric gradient, as in \cref{e:VorticityRd}. When working in 3D, it is usually more convenient to use the language of three-vectors in condition ($E$). This leads us to the condition $(E')$ in \cref{P:EquivE}.

\begin{prop}\label{P:EquivE}
    The condition (E) in \cref{T:VVEquiv} is equivalent to
    \begin{align*}
        (E') \qquad \curl u \to \curl \ol{u} + (\ol{u} \times \n) \mu
                \text{ in } L^\iny((0, T; (H^1(\Omega)^3)^*).
    \end{align*}
\end{prop}
\begin{proof}
If $A$ is an antisymmetric $3 \times 3$ matrix then
\begin{align*}
    A \cdot M
        &= \frac{A \cdot M + A \cdot M}{2}
        = \frac{A \cdot M + A^T \cdot M^T}{2}
        = \frac{A \cdot M - A \cdot M^T}{2} \\
        &= A \cdot \frac{M - M^T}{2}.
\end{align*}
Thus, since $\omega$ and $\ol{\omega}$ are antisymmetric, referring to \refE{EExplicit}, we see that ($E$) is equivalent to
\begin{align*}
        (\omega(t), M) \to (\ol{\omega}(t), M)
            - \int_\Gamma (M \n) \cdot \ol{u}(t)
                \text{ in } L^\iny([0, T])
\end{align*}
for all \textit{antisymmetric} matrices $M \in (H^1(\Omega))^{3 \times 3}$.

Now, for any three vector $\varphi$ define
\begin{align*}
    F(\varphi)
        &= \tmatrix{0 & -\varphi_3 & \varphi_2}
                   {\varphi_3 & 0 & -\varphi_1}
                   {-\varphi_2 & \varphi_1 & 0}.
\end{align*}
Then $F$ is a bijection from the vector space of three-vectors to the space of antisymmetric $3 \times 3$ matrices. Straightforward calculations show that
\begin{align*}
    F(\varphi) \cdot F(\psi)
        = 2 \varphi \cdot \psi, \qquad
    F(\varphi) v
        = \varphi \times v
\end{align*}
for any three-vectors, $\varphi$, $\psi$, $v$.
Also, $F(\curl u) = 2 \omega$ and $F(\curl \ol{u}) = 2 \ol{\omega}$.

For any $\varphi \in (H^1(\Omega))^3$ let $M = F(\varphi)$. Then
\begin{align*}
    (\omega, M)
        &= \frac{1}{2} \pr{F(\curl u), F(\varphi)}
        = \pr{\curl u, \varphi}, \\
    (\ol{\omega}, M)
        &= \frac{1}{2} \pr{F(\curl \ol{u}), F(\varphi)}
        = \pr{\curl \ol{u}, \varphi}, \\
    (M \n) \cdot \ol{u}
        &= (F(\varphi) \n) \cdot \ol{u}
        = (\varphi \times \n) \cdot \ol{u}
        = - (\ol{u} \times \n) \cdot \varphi.
\end{align*}
In the last equality, we used the scalar triple product identity $(a \times b) \cdot c = - a \cdot (c \times b)$. Because $F$ is a bijection, this gives the equivalence of ($E$) and ($E'$).
\end{proof}

\section{\texorpdfstring{$L^p$}{Lp}-norms of the vorticity blow up for \texorpdfstring{$p > 1$}{p > 1}}\label{S:LpNormsBlowUp}

\noindent

\begin{theorem}\label{T:VorticityNotBounded}
Assume that $\ol{u}$ is not identically zero on $[0, T] \times \Gamma$.
If any of the equivalent conditions of \cref{T:VVEquiv} holds then for all $p \in (1, \iny]$,
\begin{align}\label{e:omegaBlowup}
	\limsup_{\nu \to 0^+} \norm{\omega}_{L^\iny([0, T]; L^p)}
		\to \iny.
\end{align}
\end{theorem}
\begin{proof}
We prove the contrapositive. Assume that the conclusion is not true. Then for some $q' \in (1, \iny]$ it must be that for some $C_0 > 0$ and $\nu_0 > 0$,
\begin{align}\label{e:omegaBoundedCondition}
	\norm{\omega}_{L^\iny([0, T]; L^{q'})} \le C_0
		\text{ for all } 0 < \nu \le \nu_0.
\end{align}
Since $\Omega$ is a bounded domain, if \cref{e:omegaBoundedCondition} holds for some $q' \in (1, \iny]$ it holds for all lower values of $q'$ in $(1, \iny]$, so we can assume without loss of generality that $q' \in (1, \iny)$.

Let $q = q'/(q' - 1) \in (1, \iny)$ be \Holder conjugate to $q$ and $p = 2/q + 1 \in (1, 3)$. Then $p, q, q'$ satisfy the conditions of \cref{C:TraceCor} with $(p -1) q = 2$.

Applying \cref{C:TraceCor} gives, for almost all $t \in [0, T]$,
\begingroup
\allowdisplaybreaks
\begin{align*}
	&\norm{u(t) - \ol{u}(t)}_{L^p(\Gamma)}
		\le C \norm{u(t) - \ol{u}(t)}_{L^2(\Omega)}
		        ^{1 - \frac{1}{p}}
		    \norm{\grad u(t) - \grad \ol{u}(t)}_{L^{q'}(\Omega)}
		        ^{\frac{1}{p}} \\
		&\qquad
			\le C \norm{u(t) - \ol{u}(t)}_{L^2(\Omega)}
			    ^{1 - \frac{1}{p}}
			\pr{\norm{\grad u(t)}_{L^{q'}}
			    + \norm{\grad \ol{u}(t)}_{L^{q'}}}
			        ^{\frac{1}{p}} \\
		&\qquad
			\le C \norm{u(t) - \ol{u}(t)}_{L^2(\Omega)}
			    ^{1 - \frac{1}{p}}
			\pr{C(q') \norm{\omega(t)}_{L^{q'}}
			    + \norm{\grad \ol{u}(t)}_{L^{q'}}}
			        ^{\frac{1}{p}} \\
		&\qquad 
			\le C \norm{u(t) - \ol{u}(t)}_{L^2(\Omega)}
			    ^{1 - \frac{1}{p}}
\end{align*}
\endgroup
for all $0 < \nu \le \nu_0$. Here we used \cref{e:omegaBoundedCondition} and the inequality, $\norm{\grad u}_{L^{q'}(\Omega)} \le C(q') \norm{\omega}_{L^{q'}(\Omega)}$ for all $q' \in (1, \iny)$ of Yudovich \cite{Y1963}. Hence,
\begin{align*}
	\norm{u - \ol{u}}_{L^\iny([0, T]; L^p(\Gamma))}
		\le C \norm{u - \ol{u}}_{L^\iny([0, T]; L^2(\Omega))}
		    ^{1 - \frac{1}{p}}
		    \to 0
\end{align*}
as $\nu \to 0$. But,
\begin{align*}
	\norm{u - \ol{u}}_{L^\iny([0, T]; L^p(\Gamma))}
		= \norm{\ol{u}}_{L^\iny([0, T]; L^p(\Gamma))}
		\ne 0,
\end{align*}
so condition (B) cannot hold and so neither can any of the equivalent conditions in \cref{T:VVEquiv}.
\end{proof}

%
%
\section{Improved convergence when vorticity bounded in \texorpdfstring{$L^1$}{L1}}\label{S:ImprovedConvergence}

\noindent In \cref{S:LpNormsBlowUp} we showed that if the classical vanishing viscosity limit holds then the $L^p$ norms of $\omega$ must blow up as $\nu \to 0$ for all $p \in (1, \iny]$---unless the Eulerian velocity vanishes identically on the boundary. This leaves open the possibility that the $L^1$ norm of $\omega$ could remain bounded, however, and still have the classical vanishing viscosity limit. This happens, for instance, for radially symmetric vorticity in a disk (Examples 1a and 3 in \cref{S:Examples}), as shown in \cite{FLMT2008}.

In fact, as we show in \cref{C:EquivConvMeasure}, when ($VV$) holds and the $L^1$ norm of $\omega$ remains bounded in $\nu$, the convergence in condition ($E$) is stronger; namely, $weak^*$ in measure (as in \cite{FLMT2008}). (See \cref{e:RadonMeasures} and the comments after it for the definitions of $\Cal{M}(\ol{\Omega})$ and $\mu$.)

\begin{cor}\label{C:EquivConvMeasure}
    Suppose that $u \to \ol{u} \text{ in } L^\iny(0, T; H)$ and
    $\curl u$ is bounded in $L^\iny(0, T; L^1(\Omega))$ uniformly in $\eps$.
    Then in 3D,
    \begin{align}\label{e:BetterConvergence}
        \curl u \to \curl \ol{u} + (u_0 \times \n) \mu
            \quad \weak^* \text{ in } L^\iny(0, T; \Cal{M}(\ol{\Omega})).
    \end{align}
    Similarly, ($C$), ($E$), and ($E_2$) hold with $\weak^*$ convergences
    in $L^\iny(0, T; \Cal{M}(\ol{\Omega}))$ rather than uniformly in
    $(H^1(\Omega))^*$.
\end{cor}
\begin{proof}
    We prove \cref{e:BetterConvergence} explicitly for 3D solutions,
    the results for ($C$), ($E$), and ($E_2$) following in the same way.
    
    Let $\psi \in C(\ol{\Omega})$. What we must show is that
    \begin{align*}
        (\curl u(t) - \curl \ol{u}(t), \psi)
            \to \int_\Gamma (u_0(t) \times \n) \cdot \psi
                    \text{ in } L^\iny([0, T]).
    \end{align*}
    So let $\eps > 0$ and choose $\varphi \in H^1(\Omega)^d$ with
    $\norm{\psi - \varphi}_{C(\ol{\Omega})} < \eps$. We can always find
    such a $\varphi$ because $H^1(\Omega)$ is dense in $C(\ol{\Omega})$.
    Let
    \begin{align*}
        M = \max \set{\norm{\curl u
                        - \curl \ol{u}}_{L^\iny(0, T; L^1(\Omega))},
                \norm{\ol{u}}_{L^\iny([0, T] \times \Omega)}},
    \end{align*}
    which we note is finite since $\norm{\curl u}_{L^\iny(0, T; L^1(\Omega))}$
    and $\norm{\curl \ol{u}}_{L^\iny(0, T; L^1(\Omega))}$ are both
    finite. Then
    \begingroup
    \allowdisplaybreaks
    \begin{align*}
        &\abs{(\curl u(t) - \curl \ol{u}(t), \psi)
            - \int_\Gamma (u_0(t) \times \n) \cdot \psi} \\
            &\qquad
            \le \abs{(\curl u(t) - \curl \ol{u}(t), \psi - \varphi)
            - \int_\Gamma (u_0(t) \times \n) \cdot
                (\psi - \varphi)} \\
            &\qquad\qquad
            + \abs{(\curl u(t) - \curl \ol{u}(t), \varphi)
            - \int_\Gamma (u_0(t) \times \n) \cdot \varphi} \\
            &\qquad
            \le 2 M \eps
            + \abs{(\curl u(t) - \curl \ol{u}(t), \varphi)
            - \int_\Gamma (u_0(t) \times \n) \cdot \varphi}.
    \end{align*}
    \endgroup
    By \cref{P:EquivE}, we can make the last term above smaller
    than, say, $\eps$, by choosing $\nu$ sufficiently small, which is sufficient
    to give the result.
\end{proof}

\begin{remark}
Suppose that we have the slightly stronger condition that $\grad u$ is bounded in $L^\iny(0, T; L^1(\Omega))$ uniformly in $\eps$. If we are in 2D, $W^{1, 1}(\Omega)$ is compactly embedded in $L^2(\Omega)$. This is sufficient to conclude that ($VV$) holds, as shown in \cite{GKLMN14}.
\end{remark}

%
%
\section{Width of the boundary layer}\label{S:BoundaryLayerWidth}

\noindent
Working in two dimensions, make the assumptions on the initial velocity and on the forcing in  \cref{T:VVEquiv}, and assume in addition that the total mass of the initial vorticity does not vanish; that is,
\begin{align}\label{e:NonzeroMass}
	m := \int_\Omega \omega_0 = (\omega_0, 1) \ne 0.
\end{align}
(In particular, this means that $u_0$ is not in $V$.) The total mass of the Eulerian vorticity is conserved so
\begin{align}\label{e:mEAllTime}
    (\ol{\omega}(t), 1) = m \text{ for all } t \in \R.
\end{align}
The Navier-Stokes velocity, however, is in $V$ for all positive time, so its total mass is zero; that is,
\begin{align}\label{e:mNSAllTime}
    (\omega(t), 1) = 0  \text{ for all } t > 0.
\end{align}

Let us suppose that the vanishing viscosity limit holds. Fix $\delta > 0$ let $\varphi_\delta$ be a smooth cutoff function equal to $1$ on $\Gamma_\delta$ and equal to 0 on $\Omega \setminus \Gamma_{2 \delta}$. Then by ($F_2$) of \cref{T:VVEquiv} and using \cref{e:mEAllTime},
\begin{align*}
    \abs{(\omega, 1 - \varphi_\delta) - m}
        \to \abs{(\ol{\omega}, 1 - \varphi_\delta) - m}
        = \abs{m - (\ol{\omega}, \varphi_\delta) - m}
        \le C \delta,
\end{align*}
the convergence being uniform on $[0, T]$. Thus, for all sufficiently small $\nu$,
\begin{align}\label{e:omega1phiLimit}
	\abs{(\omega, 1 - \varphi_\delta) - m} \le C \delta.
\end{align}


In \cref{e:omega1phiLimit} we must hold $\delta$ fixed as we let $\nu \to 0$, for that is all we can obtain from the weak convergence in ($F_2$). Rather, this is all we can obtain without making some assumptions about the rates of convergence, a matter we will return to in the next section.


Still, it is natural to ask whether we can set $\delta = c \nu$ in \cref{e:omega1phiLimit}, this being the width of the boundary layer in Kato's seminal paper \cite{Kato1983} on the subject. If this could be shown to hold it would say that outside of Kato's layer the vorticity for solutions to ($NS$) converges in a (very) weak sense to the vorticity for the solution to ($E$). The price for such convergence, however, would be a buildup of vorticity inside the layer to satisfy the constraint in \cref{e:mNSAllTime}.

In fact, however, this is not the case, at least not by a closely related measure of vorticity buildup near the boundary.
The total mass of the vorticity (in fact, its $L^1$-norm) in any layer smaller than that of Kato goes to zero and, if the vanishing visocity limit holds, then the same holds for Kato's layer. Hence, if there is a layer in which vorticity accumulates, that layer is at least as wide as Kato's and is wider than Kato's if the vanishing viscosity limit holds. This is the content of the following theorem.

\begin{theorem}\label{T:BoundaryLayerWidth}
Make the assumptions on the initial velocity and on the forcing in \cref{T:VVEquiv}. For any positive function $\delta = \delta(\nu)$,
\begin{align}\label{e:OmegaL1VanishGeneral}
	\norm{\omega}_{L^2([0, T]; L^1(\Gamma_{\delta(\nu)}))}
		\le C \pr{\frac{\delta(\nu)}{\nu}}^{1/2}.
\end{align}
If the vanishing viscosity limit holds and
\begin{align*} 
	\limsup_{\nu \to 0^+} \frac{\delta(\nu)}{\nu} < \iny
\end{align*}
then
\begin{align}\label{e:OmegaL1Vanish}
	\norm{\omega}_{L^2([0, T]; L^1(\Gamma_{\delta(\nu)}))}
		\to 0 \text{ as } \nu \to 0.
\end{align}
\end{theorem}

\begin{proof}
By the Cauchy-Schwarz inequality,
\begin{align*}
	\norm{\omega}_{L^1(\Gamma_{\delta(\nu)})}
		\le \norm{1}_{L^2(\Gamma_{\delta(\nu)})} \norm{\omega}_{L^2(\Gamma_{\delta(\nu)})}
			\le C \delta^{1/2} \norm{\omega}_{L^2(\Gamma_{\delta(\nu)})}
\end{align*}
so
\begin{align*}
	\frac{C}{\delta} \norm{\omega}_{L^1(\Gamma_{\delta(\nu)})}^2
	    \le \norm{\omega}_{L^2(\Gamma_{\delta(\nu)})}^2
\end{align*}
and
\begin{align*}
	\frac{C \nu}{\delta} \norm{\omega}_{L^2([0, T]; L^1(\Gamma_{\delta(\nu)}))}^2
	    \le \nu \norm{\omega}_{L^2([0, T]; L^2(\Gamma_{\delta(\nu)}))}^2.
\end{align*}
By the basic energy inequality for the Navier-Stokes equations, the right-hand side is bounded, giving \refE{OmegaL1VanishGeneral}, and if the vanishing viscosity limit holds, the right-hand side goes to zero by \cref{e:KellCondition}, giving \refE{OmegaL1Vanish}.
\end{proof}

\begin{remark}
In \refT{BoundaryLayerWidth}, we do not need the assumption in \refE{NonzeroMass} nor do we need to assume that we are in dimension two. The result is of most interest, however, when one makes these two assumptions.
\end{remark}

\begin{remark}
\refE{OmegaL1Vanish} also follows from condition (iii'') in \cite{K2006Kato} using the Cauchy-Schwarz inequality in the manner above, but that is using a sledge hammer to prove a simple inequality. Note that \refE{OmegaL1Vanish} is necessary for the vanishing viscosity limit to hold, but is not (as far as we can show) sufficient.
\end{remark}


%
%
\section{Optimal convergence rate}\label{S:OptimalConvergenceRate}

\noindent Still working in two dimensions, let us return to \cref{e:omega1phiLimit}, assuming as in the previous section that the vanishing viscosity limit holds, but bringing the rate of convergence function, $F$, of \cref{T:ROC} into the analysis. We will now make $\delta = \delta(\nu) \to 0$ as $\nu \to 0$, and choose $\varphi_\delta$ slightly differently, requiring that it equal $1$ on $\Gamma_{\delta^*}$ and vanish outside of $\Gamma_\delta$ for some $0 < \delta^* = \delta^*(\nu) < \delta$. We can see from the argument that led to \cref{e:omega1phiLimit}, incorporating the convergence rate for ($F_2$) given by \cref{T:ROC}, that
\begin{align*}
    \abs{(\omega, 1 - \varphi_\delta) - m}
        \le C \delta + \norm{\grad \varphi_\delta}_{L^2(\Omega)} F(\nu).
\end{align*}
Because $\prt \Omega$ is $C^2$, we can always choose $\varphi_\delta$ so that $\abs{\grad \varphi_\delta} \le C(\delta - \delta^*)^{-1}$. Then for all sufficiently small $\delta$,
\begin{align*}
    \norm{\grad \varphi_\delta}_{L^2(\Omega)}
        \le \pr{\int_{\Gamma_\delta \setminus \Gamma_{\delta^*}}
                \pr{\frac{C}{\delta - \delta^*}}^2}^{\frac{1}{2}}
        = C \frac{(\delta - \delta^*)^{\frac{1}{2}}}{\delta - \delta^*}
        = C (\delta - \delta^*)^{-\frac{1}{2}}.
\end{align*}
We then have
\begin{align}\label{e:mDiffEst}
    \abs{(\omega, 1 - \varphi_\delta) - m}
        \le C \brac{\delta + (\delta - \delta^*)^{-\frac{1}{2}} F(\nu)}.
\end{align}

For any measurable subset $\Omega'$ of $\Omega$, define
\begin{align*}
    \mathbf{M}(\Omega') = \int_{\Omega'} \omega,
\end{align*}
the total mass of vorticity on $\Omega'$. Then
\begin{align*}
    \mathbf{M}(\Gamma_\delta^C)
        = (\omega, 1 - \varphi_\delta)
            + \int_{\Gamma_\delta \setminus \Gamma_{\delta^*}} \varphi_\delta \omega
\end{align*}
so
\begin{align}\label{e:MDiffEst}
    \begin{split}
    \abs{(\omega, 1 - \varphi_\delta) - \mathbf{M}(\Gamma_\delta^C)}
        &\le \norm{\omega}_{L^2(\Gamma_\delta \setminus \Gamma_{\delta^*})}
            \norm{\varphi_\delta}_{L^2(\Gamma_\delta \setminus \Gamma_{\delta^*})} \\
        &\le C (\delta - \delta^*)^{\frac{1}{2}}
            \norm{\omega}_{L^2(\Gamma_{\delta})}.
    \end{split}
\end{align}

\Ignore{ 
To obtain any reasonable control on the total mass of vorticity, we certainly need $\delta, \delta^* \to 0$ as $\nu \to 0$, but more important, as we can see from \cref{e:mDiffEst}, we need
\begin{align}\label{e:LayerReq1}
    (\delta - \delta^*)^{-\frac{1}{2}} F(\nu) \to 0
        \text{ as } \nu \to 0.
\end{align}
In light of \cref{T:BoundaryLayerWidth} and its proof, we should also require at least that
\begin{align}\label{e:LayerReq2}
    (\delta - \delta^*)^{\frac{1}{2}}
            \norm{\omega}_{L^2(0, T; L^2(\Gamma_\delta))} \to 0
        \text{ as } \nu \to 0
\end{align}
so that the bound in \cref{e:mDiffEst} will lead, via \cref{e:MDiffEst}, to a bound on the total mass of vorticity outside the boundary layer, $\Gamma_\delta$.

Now, as in the proof of \cref{T:BoundaryLayerWidth}, if we let $\delta - \delta^* = O(\nu)$ then the condition in \cref{e:LayerReq2} will hold by \cref{e:KellCondition}. Then the requirement in \cref{e:LayerReq1} becomes \begin{align*}
    F(\nu)
        = o \pr{(\delta - \delta^*)^{\frac{1}{2}}}
        = o (\nu^{\frac{1}{2}}).
\end{align*}
} 

From these observations and those in the previous section, we have the following:
\begin{theorem}\label{T:VorticityMassControl}
    \Ignore{ 
    Assume that $\delta = \delta(\nu) \to 0$ as $\nu \to 0$ and define
    \begin{align*}
        M_\delta
            = \norm{\int_\Omega \omega_0
                - \int_{\Gamma_\delta^C} \omega(t)}_{L^2([0, T])}.
    \end{align*}
    If the classical vanishing viscosity limit in ($VV$) holds with a rate that is
    $o(\nu^{\frac{1}{2}})$ then $M_{\delta(\nu)} \to 0$ as $\nu \to 0$.
    } 
    Assume that the classical vanishing viscosity limit in ($VV$) holds with a rate
    of convergence,
    $F(\nu) = o(\nu^{1/2})$. Then in 2D
    the initial mass of the vorticity must be zero.
\end{theorem}
\begin{proof}
    From \cref{e:mDiffEst,e:MDiffEst},
    \begin{align*}
        M_\delta
            &:= \abs{m - \mathbf{M}(\Gamma_\delta^C)}
            \le \abs{m - (\omega, 1 - \varphi_\delta)}
                + \abs{(\omega, 1 - \varphi_\delta) - \mathbf{M}(\Gamma_\delta^C)} \\
            &\le C \brac{\delta + (\delta - \delta^*)^{-\frac{1}{2}} F(\nu)}
                + C (\delta - \delta^*)^{\frac{1}{2}}
            \norm{\omega}_{L^2(\Gamma_{\delta})}.
    \end{align*}
    Choosing $\delta(\nu) = \nu$, $\delta^*(\nu) = \nu/2$, we have
    \begin{align*}
        M_\nu
            &\le C \brac{\nu + \nu^{-\frac{1}{2}} o(\nu^{\frac{1}{2}})}
                + C \nu^{\frac{1}{2}}
            \norm{\omega}_{L^2(\Gamma_{\nu})},
    \end{align*}
    uniformly over $[0, T]$. Squaring, integrating in time, and applying Young's
    inequality gives
    \begin{align*}
        \norm{M_\nu}_{L^2([0, T])}^2
            = \int_0^T M_\nu^2
            \le CT (\nu^2 + o(1))
                + C \nu \int_0^T \norm{\omega}_{L^2(0, T; L^2(\Gamma_\nu))}^2
                \to 0
    \end{align*}
    as $\nu \to 0$ by \cref{e:KellCondition}.
    
    Then,
    \begin{align*}
        \norm{m - M(\Omega)}_{L^2([0, T])}
            &\le \norm{m - M(\Gamma_\nu^C)}_{L^2([0, T])}
                + \norm{M(\Gamma_\nu)}_{L^2([0, T])} \\
            &\le \norm{M_\nu}_{L^2([0, T])}
                + \norm{\omega}_{L^2([0, T]; L^1(\Gamma_{\nu}))}
            \to 0
    \end{align*}
    as $\nu \to 0$ by \cref{T:BoundaryLayerWidth}.
    
    But $u(t)$ lies in $V$ so $M(\Omega) = 0$ for all $t > 0$.
    Hence, the limit above is possible only if $m = 0$.
\end{proof}

For non-compatible initial data, that is for $u_0 \notin V$, the total mass of vorticity will generically not be zero, so $C \sqrt{\nu}$ should be considered a bound on the rate of convergence for non-compatible initial data. As we will see in \cref{R:ROC}, however, a rate of convergence as good as $C \sqrt{\nu}$ is almost impossible unless the initial data is fairly smooth, and even then it would only occur in special circumstances.

Therefore, let us assume that the rate of convergence in ($VV$) is only $F(\nu) = C \nu^{1/4}$. As we will see in \cref{S:Examples}, this is a more typical rate of convergence for the simple examples for which ($VV$) is known to hold.
Now \cref{e:mDiffEst} still gives a useful bound as long as $\delta - \delta^*$ is slightly larger than the Prandtl layer width of $C \sqrt{\nu}$ (though \cref{e:MDiffEst} then fails to tell us anything useful). So let us set $\delta = 2 \nu^{1/2 - \eps}$, $\delta^* = \nu^{1/2 - \eps}$, $\eps > 0$ arbitrarily small. We are building here to a conjecture, so for these purposes we will act as though $\eps = 0$.

If the Prandtl theory is correct, then we should expect that $\mathbf{M}(\Gamma_\delta^C) \to m$ as $\nu \to 0$, since outside of the Prandtl layer $u$ matches $\ol{u}$. But the total mass of vorticity for all positive time is zero, and the total mass in the Kato Layer, $\Gamma_\nu$, goes to zero by \cref{T:BoundaryLayerWidth}. There would be no choice then but to have a total mass of vorticity between the Kato and Prandtl layers that approaches $-m$ as the viscosity vanishes. (Since the Kato layer is much smaller than the Prandtl layer, this does not require that there be any higher concentration of vorticity in any particular portion of the Prandtl layer, though.)

Now suppose that the rate of convergence is even slower than $C \nu^{1/4}$. Then \cref{e:mDiffEst}  gives a measure of $\mathbf{M}(\Gamma_\delta^C)$ converging to $m$ well outside the Prandtl layer. This does not directly contradict any tenet of the Prandtl theory, but it suggests that for small viscosity the solution to the Navier-Stokes equations matches the solution to the Euler equations only well outside the Prandtl layer. This leads us to the following conjecture:

\begin{conj}\label{J:Prandtl}
    If the vanishing viscosity limit in ($VV$) holds at a rate slower than
    $C \nu^{\frac{1}{4}}$ in 2D then the Prandtl theory fails.
\end{conj}

We conjecture no further, however, as to whether the Prandtl equations become ill-posed or whether the formal asymptotics fail to hold rigorously.

%
%
\section{Some kind of convergence always happens}\label{S:SomeConvergence}

\noindent Assume that $v$ is a vector field lying in  $L^\iny([0, T]; H^1(\Omega))$. An examination of the proof given in \cite{K2008VVV} of the chain of implications in \cref{T:VVEquiv} shows that all of the conditions except (B) are still equivalent with $\ol{u}$ replaced by $v$. That is, defining
\begingroup
\allowdisplaybreaks
\begin{align*}
	(A_v) & \qquad u \to v \text{ weakly in } H
					\text{ uniformly on } [0, T], \\
	(A'_v) & \qquad u \to v \text{ weakly in } (L^2(\Omega))^d
					\text{ uniformly on } [0, T], \\
	(B_v) & \qquad u \to v \text{ in } L^\iny([0, T]; H), \\
	(C_v) & \qquad \grad u \to \grad v - \innp{\gamma_\mathbf{n} \cdot, v \mu} 
				\text{ in } ((H^1(\Omega))^{d \times d})^*
					   \text{ uniformly on } [0, T], \\
	(D_v) & \qquad \grad u \to \grad v \text{ in } (H^{-1}(\Omega))^{d \times d}
					   \text{ uniformly on } [0, T], \\
	(E_v) & \qquad \omega \to \omega(v)
					- \frac{1}{2} \innp{\gamma_\mathbf{n} (\cdot - \cdot^T),
								v \mu}
					\text{ in } 
	 				((H^1(\Omega))^{d \times d})^*
	 				   \text{ uniformly on } [0, T], \\
	(E_{2, v}) & \qquad \omega \to \omega(v) - (v \cdot \BoldTau) \mu 
				\text{ in } (H^1(\Omega))^*
					   \text{ uniformly on } [0, T], \\
	(F_{2, v}) & \qquad \omega \to \omega(v) \text{ in } H^{-1}(\Omega)
					   \text{ uniformly on } [0, T],
\end{align*}
\endgroup
we have the following theorem:
\begin{theorem}\label{T:MainResultv}
	Conditions ($A_v$), ($A'_v$), ($C_v$), ($D_v$), and ($E_v$) are equivalent.
	In 2D, conditions ($E_{2,v}$) and, when $\Omega$ is simply connected,
	($F_{2,v}$) are equivalent to the other conditions.
	Also, $(B_v)$ implies all of the other conditions. Finally,
	the same equivalences hold if we replace each
	convergence above with the convergence of a subsequence.
\end{theorem}

But we also have the following:
\begin{theorem}\label{T:SubsequenceConvergence}
There exists $v$ in $L^\iny([0, T]; H)$ such that a subsequence $(u_\nu)$ converges weakly to $v$ in $L^\iny([0, T]; H)$.
\end{theorem}
\begin{proof}
The argument for a simply connected domain in 2D is slightly simpler so we give it first. 
The sequence $(u_\nu)$ is bounded in $L^\iny([0, T]; H)$ by the basic energy inequality for the Navier-Stokes equations. Letting $\psi_\nu$ be the stream function for $u_\nu$ vanishing on $\Gamma$, it follows by the Poincare inequality that $(\psi_\nu)$ is bounded in $L^\iny([0, T]; H_0^1(\Omega))$. Hence, there exists a subsequence, which we relabel as $(\psi_\nu)$, converging strongly in $L^\iny([0, T]; L^2(\Omega))$ and weak-* in $L^\iny([0, T]; H_0^1(\Omega))$ to some $\psi$ lying in $L^\iny([0, T]; H_0^1(\Omega))$. Let $v = \grad^\perp \psi$.

Let $g$ be any element of $L^\iny([0, T]; H)$. Then
\begin{align*}
	(u_\nu, g)
		&= (\grad^\perp \psi_\nu, g)
		= - (\grad \psi_\nu, g^\perp)
		= (\psi_\nu, - \dv g^\perp)
		= (\psi_\nu, \omega(g)) \\
		&\to (\psi, \omega(g))
		= (v, g).
\end{align*}
In the third equality we used the membership of $\psi_v$ in $H_0^1(\Omega)$ and the last equality follows in the same way as the first four. The convergence follows from the weak-* convergence of $\psi_\nu$ in in $L^\iny([0, T]; H_0^1(\Omega))$ and the membership of $\omega(g)$ in $H^{-1}(\Omega)$. 

In dimension $d \ge 3$, let $M_\nu$ in $(H_0^1(\Omega))^d$ satisfy $u_\nu = \dv M_\nu$; this is possible by Corollary 7.5 of \cite{K2008VVV}. Arguing as before it follows that there exists a subsequence, which we relabel as $(M_\nu)$, converging strongly in $L^\iny([0, T]; L^2(\Omega))$ and weak-* in $L^\iny([0, T]; H_0^1(\Omega))$ to some $M$ that lies in $L^\iny([0, T]; (H_0^1(\Omega))^{d \times d})$. Let $v = \dv M$.

Let $g$ be any element of $L^\iny([0, T]; H)$. Then
\begin{align*}
	(u_\nu, g)
		&= (\dv M_\nu, g)
		= -(M_\nu, \grad g)
		\to - (M, \grad g)
		= (v, g),
\end{align*}
establishing convergence as before.
\end{proof}

It follows from \refTAnd{MainResultv}{SubsequenceConvergence} that all of the convergences in \cref{T:VVEquiv} hold except for $(B)$, but for a subseqence of solutions and the convergence is to some velocity field $v$ lying only in $L^\iny([0, T]; H)$ and not necessarily in $L^\iny([0, T]; H \cap H^1(\Omega))$ . In particular, we do not know if $v$ is a solution to the Euler equations, and, in fact, there is no reason to expect that it is.


\Ignore{ 
%
%
\section{Convergence to another solution to the Euler equations?}

\noindent One could imagine that the solutions $u = u_\nu$ to the Navier-Stokes equations converge, in the limit, to a solution to the Euler equations, but one different from $\ol{u}$ and possibly with lower regularity. Since such solutions are determined by their initial velocity, this means that the vector $v$ to which $(u_\nu)$ converges has initial velocity $v^0 \ne \ol{u}^0$. (This conclusion would be true even if $v$ had so little regularity that it had not been determined uniquely by its initial velocity.)

Now, $\ol{u}(t)$ is continuous in $H$, since it is a strong solution, as too, if we restrict ourselves to two dimensions, is $u(t)$. If $v$ has bounded vorticity, say, then $v(t)$ is continuous in $H$ as well. It would seem ......
} 


\newpage
%
%
\Part{Theme II: Kato's Conditions}

%
%
\section{An equivalent 2D condition on the boundary}\label{S:EquivCondition}

\noindent

\begin{theorem}\label{T:BoundaryIffCondition}
For ($VV$) to hold in 2D it is necessary and sufficient that
\begin{align}\label{e:BoundaryCondition2D}
	\nu \int_0^T \int_\Gamma \omega \, \ol{u} \cdot \BoldTau
		\to 0
		\text{ as } \nu \to 0.
\end{align}
\end{theorem}
\begin{proof}
Since the solution is in 2D and $f \in L^2(0, T; H) \supseteq C^1_{loc}(\R; C^1(\Omega))$, Theorem III.3.10 of \cite{T2001} gives
\begin{align}\label{e:RegTwoD}
    \begin{split}
        &\sqrt{t} u \in L^2(0, T; H^2(\Omega)) \cap L^\iny(0, T; V), \\
        &\sqrt{t} \prt_t u \in L^2(0, T; H),
    \end{split}
\end{align}
so $\omega(t)$ is defined in the sense of a trace on the boundary. This shows that the condition in \cref{e:BoundaryCondition2D} is well-defined.

For simplicity we give the argument with $f = 0$. We perform the calculations using the $d$-dimensional form of the vorticity in \cref{e:VorticityRd}, specializing to 2D only at the end. (The argument applies formally in higher dimensions; see \cref{R:BoundaryConditionInRd}.)

Subtracting ($EE$) from ($NS$), multiplying by $w = u - \ol{u}$, integrating over $\Omega$, using \cref{L:TimeDerivAndIntegration} for the time derivative, and $u(t) \in H^2(\Omega)$, $t > 0$, for the spatial integrations by parts, leads to
\begin{align}\label{e:BasicEnergyEq}
    \begin{split}
	\frac{1}{2} \diff{}{t} &\norm{w}_{L^2}^2
		+ \nu \norm{\grad u}_{L^2}^2 \\
		&= - (w \cdot \grad \ol{u}, w)
			+ \nu(\grad u, \grad \ol{u})
				- \nu \int_\Gamma (\grad u \cdot \mathbf{n}) \cdot \ol{u}.
	\end{split}
\end{align}

Now,
\begin{align*}
    \begin{split}
    (\grad u \cdot \mathbf{n}) \cdot \ol{u}
        &= 2 (\frac{\grad u - (\grad u)^T}{2} \cdot \mathbf{n})
                \cdot \ol{u}
            + ((\grad u)^T \cdot \n) \cdot \ol{u} \\
        &= 2 (\omega(u) \cdot \mathbf{n})
                \cdot \ol{u}
            + ((\grad u)^T \cdot \n) \cdot \ol{u}.
    \end{split}
\end{align*}
But,
\begin{align*}
    \int_\Gamma &((\grad u)^T \cdot \n) \cdot \ol{u}
        = \int_\Gamma \prt_i u^j n^j \ol{u}^i
        = \frac{1}{2} \int_\Gamma \prt_i(u \cdot \n) \ol{u}^i \\
        &= \frac{1}{2} \int_\Gamma \grad (u \cdot \n) \cdot \ol{u}
        = 0,
\end{align*}
since $u \cdot \n = 0$ on $\Gamma$ and $\ol{u}$ is tangent to $\Gamma$. Hence,
\begin{align}\label{e:gradunuolEq}
    \int_\Gamma (\grad u \cdot \mathbf{n}) \cdot \ol{u}
        = 2 (\omega(u) \cdot \mathbf{n})
                \cdot \ol{u}
\end{align}
and
\begin{align*}
	\frac{1}{2} \diff{}{t} &\norm{w}_{L^2}^2
		+ \nu \norm{\grad u}_{L^2}^2 \\
		&= - (w \cdot \grad \ol{u}, w)
			+ \nu(\grad u, \grad \ol{u})
				- 2 \nu \int_\Gamma (\omega(u) \cdot \mathbf{n})
                \cdot \ol{u}.
\end{align*}

By virtue of \cref{L:TimeDerivAndIntegration}, we can integrate over time to give
\begin{align}\label{e:VVArg}
	\begin{split}
	&\norm{w(T)}_{L^2}^2
		+ 2 \nu \int_0^T \norm{\grad u}_{L^2}^2
		= - 2 \int_0^T (w \cdot \grad \ol{u}, w)
			+ 2 \nu \int_0^T (\grad u, \grad \ol{u}) \\
		&\qquad - 2 \nu \int_0^T \int_\Gamma (\omega(u)
		    \cdot \mathbf{n}) \cdot \ol{u}.
	\end{split}
\end{align}

In two dimensions, we have (see (4.2) of \cite{KNavier})
\begin{align}\label{e:gradunomega}
	(\grad u \cdot \mathbf{n}) \cdot \ol{u}
		= ((\grad u \cdot \mathbf{n}) \cdot \BoldTau) (\ol{u} \cdot \BoldTau)
		= \omega(u) \, \ol{u} \cdot \BoldTau,
\end{align}
and \cref{e:VVArg} can be written
\begin{align}\label{e:VVArg2D}
	\begin{split}
	&\norm{w(T)}_{L^2}^2
		+ 2 \nu \int_0^T \norm{\grad u}_{L^2}^2
		= - 2 \int_0^T (w \cdot \grad \ol{u}, w)
			+ 2 \nu \int_0^T (\grad u, \grad \ol{u}) \\
		&\qquad -  \nu \int_0^T \int_\Gamma \omega(u) \, \ol{u} \cdot \BoldTau.
	\end{split}
\end{align}

The sufficiency of \refE{BoundaryCondition2D} for the vanishing viscosity limit ($VV$) to hold (and hence for  the other conditions in \cref{T:VVEquiv} to hold) follows from the bounds,
\begin{align*}
	\abs{(w \cdot \grad \ol{u}, w)}
		&\le \norm{\grad \ol{u}}_{L^\iny([0, T] \times \Omega)}
			\norm{w}_{L^2}^2
			\le C \norm{w}_{L^2}^2, \\
		\nu \int_0^T \abs{(\grad u, \grad \ol{u})}
			&\le \sqrt{\nu} \norm{\grad \ol{u}}_{L^2([0, T] \times \Omega)}
			\sqrt{\nu} \norm{\grad u}_{L^2([0, T] \times \Omega)}
				\le C \sqrt{\nu},
\end{align*}
and Gronwall's inequality.

Proving the necessity of \refE{BoundaryCondition2D} is just as easy. Assume that $(VV)$ holds, so that $\norm{w}_{L^\iny([0, T]; L^2(\Omega))} \to 0$. Then by the two inequalities above, the first two terms on the right-hand side of \refE{VVArg2D} vanish with the viscosity as does the first term on the left-hand side. The second term on the left-hand side vanishes as proven in \cite{Kato1983} (it follows from a simple argument using the energy equalities for ($NS$) and ($E$)). It follows that, of necessity, \refE{BoundaryCondition2D} holds.
\end{proof}

\begin{remark}\label{R:ROC}
    It follows from the proof of \refT{BoundaryIffCondition} that in 2D,
    \begin{align*}
        \norm{u(t) - \ol{u}(t)}
            \le C \brac{\nu^{\frac{1}{4}}
                + \abs{\nu \int_0^T \int_\Gamma \omega \, \ol{u}
                    \cdot \BoldTau}^{\frac{1}{2}}} e^{C t}.
    \end{align*}
    Suppose that $\ol{u}_0$ is smooth enough that
    $\Delta \ol{u} \in L^\iny([0, T] \times \Omega)$.
    Then before integrating to obtain \cref{e:BasicEnergyEq} we
    can replace the term $\nu (\Delta u, w)$ with
    $\nu (\Delta w, w) + \nu (\Delta \ol{u}, w)$.
    Integrating by parts gives
    \begin{align*}
        \nu (\Delta w, w)
            = \nu \norm{\grad w}_{L^2}^2,
    \end{align*}
    and we also have,
    \begin{align*}
        \nu (\Delta \ol{u}, w)
            \le \nu \norm{\Delta \ol{u}}_{L^2} \norm{w}_{L^2}
            \le \frac{\nu^2}{2} \norm{\Delta \ol{u}}_{L^2}^2
                + \frac{1}{2} \norm{w}_{L^2}^2.
    \end{align*}
    This leads to the bound,
    \begin{align*}
        \norm{u(t) - \ol{u}(t)}_{L^2}
            \le C \brac{\nu
                + \abs{\nu \int_0^T \int_\Gamma \omega \, \ol{u}
                    \cdot \BoldTau}^{\frac{1}{2}}} e^{C t}
    \end{align*}
    (and also $\norm{u - \ol{u}}_{L^2(0, T; H^1)} \le C \nu^{1/2} e^{Ct}$).
    Thus, the bound we obtain on the rate of convergence in $\nu$ is never better
    than $O(\nu^{1/4})$
    unless the initial data is smooth enough, in which case it is never better
    than $O(\nu)$. In any case, only in exceptional circumstances would the rate
    not be determined by the integral coming from the boundary term.
\end{remark}

\begin{remark}\label{R:BoundaryConditionInRd}
Formally, the argument in the proof of \cref{T:BoundaryIffCondition} would give in any dimension the condition
\begin{align*} 
    \nu \int_0^T \int_\Gamma (\omega(u) \cdot \mathbf{n})
                \cdot \ol{u}
            \to 0
            \text{ as } \nu \to 0.
\end{align*}
In 3D, one has $\omega(u) \cdot \n = (1/2) \vec{\omega} \times \n$, so the condition could be written
\begin{align*} 
    \nu \int_0^T \int_\Gamma (\vec{\omega} \times \n)
                \cdot \ol{u}
        = \nu \int_0^T \int_\Gamma \vec{\omega} \cdot
                (\ol{u} \times \n)
            \to 0
            \text{ as } \nu \to 0,
\end{align*}
where $\vec{\omega}$ is the 3-vector form of the curl of $u$. We can only be assured, however, that $u(t) \in V$ for all $t > 0$, which is insufficient to define $\vec{\omega}$ on the boundary. (The normal component could be defined, though, since both $\vec{\omega}(t)$ and $\dv \vec{\omega}(t) = 0$ lie in $L^2$.) Even assuming more compatible initial data in 3D, such as $u_0 \in V$, we can only conclude that $u(t) \in H^2$ for a short time, with that time decreasing to $0$ as $\nu \to 0$ (in the presence of forcing; see, for instance, Theorem 9.9.4 of \cite{FoiasConstantin1988}).

\end{remark}

\Ignore{ 
\begin{remark}\label{R:BoundaryCondition2DRd}
Since $\ol{u} \times \n$ is a tangent vector, the second form of the condition in \refE{BoundaryCondition3D} shows that it is only the tangential components of $\vec{\omega}$ that matter in this condition. More specifically, only the tangential component perpendicular to $\ol{u}$ matters.

\end{remark}
} 

There is nothing deep about the condition in \refE{BoundaryCondition2D}, but what it says is that there are two mechanisms by which the vanishing viscosity limit can hold: Either the blowup of $\omega$ on the boundary happens slowly enough that
\begin{align}\label{e:nuL1Bound}
	\nu \int_0^T \norm{\omega}_{L^1(\Gamma)}
		\to 0
		\text{ as } \nu \to 0
\end{align}
or the vorticity for ($NS$) is generated on the boundary in such a way as to oppose the sign of $\ol{u} \cdot \BoldTau$. (This latter line of reasoning is followed in \cite{CKV2014}, leading to a new condition in a boundary layer slightly thicker than that of Kato.) In the second case, it could well be that vorticity for $(NS)$ blows up fast enough that \refE{nuL1Bound} does not hold, but cancellation in the integral in \refE{BoundaryCondition2D} allows that condition to hold.

\begin{lemma}\label{L:TimeDerivAndIntegration}
    Assume that $v \in L^\iny(0, T; V)$ with $\prt_t v \in L^2(0, T; V')$ as well as
    $\sqrt{t} \prt_t v \in L^2(0, T; H)$.
    Then $v \in C([0, T]; H)$,
    \begin{align*}
        \frac{1}{2} \diff{}{t} \norm{v}_{L^2}^2
            = (\prt_t v, v) \text{ in } \Cal{D}'((0, T))
        \text{ with } \sqrt{t} (\prt_t v, v) \in L^1(0, T),
    \end{align*}
    and
    \begin{align*}
        \int_0^T \diff{}{t} \norm{v(t)}_{L^2}^2 \, dt
            = \norm{v(T)}_{L^2}^2 - \norm{v(0)}_{L^2}^2.
    \end{align*}  
\end{lemma}
\begin{proof}
Having $v \in L^2(0, T; V)$ with $\prt_t v \in L^2(0, T; V')$ is enough to conclude that $(\prt_t v, v) = (1/2) (d/dt) \norm{v}_{L^2}^2$ in $\Cal{D}'((0, T))$ and $v \in C([0, T]; H)$ (see Lemma III.1.2 of \cite{T2001}).

Let $T_0 \in (0, T)$. Our stronger assumptions also give $(d/dt) \norm{v}_{L^2}^2 = 2(\prt_t v, v) \in L^1(T_0, T)$. Hence, by the fundamental theorem of calculus for Lebesgue integration (Theorem 3.35 of \cite{Folland1999}) it follows that
\begin{align*}
    \int_{T_0}^T \diff{}{t} \norm{v}_{L^2}^2 \, dt
        = \norm{v(T)}_{L^2}^2 - \norm{v(T_0)}_{L^2}^2.
\end{align*}
But $v$ is continuous in $H$ down to time zero, so taking $T_0$ to 0 completes the proof.
\end{proof}

%
%
\section{Examples where the 2D boundary condition holds}\label{S:Examples}

\noindent All examples where the vanishing viscosity limit is known to hold have some kind of symmetry---in geometry of the domain or the initial data---or have some degree of analyticity.

Since \refE{BoundaryCondition2D} is a necessary condition, it holds for all of these examples. But though it is also a sufficient condition, it is not always practicable to apply it to establish the limit. We give here examples in which it is practicable. This includes all known 2D examples having symmetry. In all explicit cases, the initial data is a stationary solution to the Euler equations.

\Example{1} Let $\ol{u}$ be any solution to the Euler equations for which $\ol{u} = 0$ on the boundary. The integral in \refE{BoundaryCondition2D} then vanishes for all $\nu$. From \refR{ROC}, the rate of convergence (here, and below, in $\nu$) is $C \nu^{1/4}$ or, for smoother initial data, $C \nu$.

\Example{1a} Example 1 is not explicit, since we immediately encounter the question of what (nonzero) examples of such steady solutions there are. As a first example, let $D$ be the disk of radius $R > 0$ centered at the origin and let $\omega_0 \in L^\iny(D)$ be radially symmetric. Then the associated velocity field, $u_0$, is given by the Biot-Savart law. By exploiting the radial symmetry, $u_0$ can be written,
\begin{align}\label{e:u0Circular}
    u_0(x)
        &= \frac{x^\perp}{\abs{x}^2}
            \int_0^{\abs{x}} \omega_0(r) r \, dr, \quad
\end{align}
where $B({\abs{x}})$ is the ball of radius $\abs{x}$ centered at the origin and where we abuse notation a bit in the writing of $\omega_0(r)$. Since $u_0$ is perpendicular to $\grad u_0$ it follows from the vorticity form of the Euler equations that $\ol{u} \equiv u_0$ is a stationary solution to the Euler equations.

Now assume that the total mass of vorticity,
\begin{align}\label{e:m}
    m := \int_{\R^2} \omega_0,
\end{align}
is zero. We see from \refE{u0Circular} that on $\Gamma$,
$u_0 = m x^\perp R^{-1} = 0$,
giving a steady solution to the Euler equations with velocity vanishing on the boundary.

(Note that $m = 0$ is equivalent to $u_0$ lying in the space $V$ of divergence-free vector fields vanishing on the boundary.)

\Example{1b} Let $\omega_0 \in L^1 \cap L^\iny(\R^2)$ be a compactly supported radially symmetric initial vorticity for which the total mass of vorticity vanishes; that is, $m = 0$. Then the expression for $u_0$ in \refE{u0Circular}, which continues to hold throughout all of $\R^2$, shows that $u_0$ vanishes outside of the support of its vorticity.

If we now restrict such a radially symmetric $\omega_0$ so that its support lies inside a domain (even allowing the support of $\omega_0$ to touch the boundary of the domain) then the velocity $u_0$ will vanish on the boundary. In particular, $u_0 \cdot \n = 0$ so, in fact, $u_0$ is a stationary solution to the Euler equations in the domain, being already one in the whole plane. In fact, one can use a superposition of such radially symmetric vorticities, as long as their supports do not overlap, and one will still have a stationary solution to the Euler equations whose velocity vanishes on the boundary.

Such a superposition is called a \textit{superposition of confined eddies} in \cite{FLZ1999A}, where their properties in the full plane, for lower regularity than we are considering, are analyzed. These superpositions provide a fairly wide variety of examples in which the vanishing viscosity limit holds. It might be interesting to investigate the precise manner in which the vorticity converges in the vanishing viscosity limit; that is, whether it is possible to do better than the ``vortex sheet''-convergence in condition $(E_2)$ of \cite{K2008VVV}.

In \cite{Maekawa2013}, Maekawa considers initial vorticity supported away from the boundary in a half-plane. We note that the analogous result in a disk, even were it shown to hold, would not cover this Example 1b when the support of the vorticity touches the boundary.

\Example{2 [2D shear flow]} Let $\phi$ solve the heat equation,
\begin{align}\label{e:HeatShear}
    \left\{
        \begin{array}{rl}
            \prt_t \phi(t, z) = \nu \prt_{zz} \phi(t, z)
                & \text{on } [0, \iny) \times [0, \iny), \\
            \phi(t, 0) = 0
                & \text{ for all } t > 0, \\
        \phi(0) = \phi_0. &
        \end{array}
    \right.  
\end{align}
Assume for simplicity that $\phi_0 \in W^{1, \iny}((0, \iny)$.
Let $u_0 = (\phi_0, 0)$ and $u(t, x) = (\phi(t, x_2), 0)$.

Let $\Omega = [-L, L] \times (0, \iny)$ be periodic in the $x_1$-direction. Then $u_0 \cdot \n = 0$ and $u(t) = 0$ for all $t > 0$ on $\prt \Omega$ and
\begin{align*}
    \prt_t u(t, x)
        &= \nu(\prt_{x_2 x_2} \phi(t, x_2), 0)
        = \nu \Delta u(t, x), \\
    (u \cdot \grad u)(t, x)
        &=
            \matrix{\prt_1 u^1 & \prt_1 u^2}
                   {\prt_2 u^1 & \prt_2 u^2}
            \matrix{u^1}{u^2}
        =
            \matrix{0 & 0}{\prt_2 \phi(t, x_2) & 0}
            \matrix{\phi(t, x_2)}{0} \\
        &=
            \matrix{0}{\prt_2 \phi(t, x_2) \phi(t, x_2)}
        =
            \frac{1}{2} \grad \phi(t, x_2).
\end{align*}
It follows that $u$ solves the Navier-Stokes equations on $\Omega$ with pressure, $p = - \frac{1}{2} \phi(t, x_2)$.

Similarly, letting $\ol{u} \equiv u_0$, we have $\prt_t \ol{u} = 0$, $\ol{u} \cdot \grad \ol{u} = \frac{1}{2} \grad \phi_0$ so $\ol{u} \equiv u_0$ is a stationary solution to the Euler equations.

Now, $\omega = \prt_1 u^2 - \prt_2 u^1 = - \prt_2 \phi(t, x_2)$ so
\begin{align*}
    \int_\Gamma \omega \, \ol{u} \cdot \BoldTau
        &= - \int_\Gamma \prt_2 \phi(t, x_2)|_{x_2 = 0}
                \phi_0(0)
        = - \phi_0(0)\int_{-L}^L
            \prt_{x_2} \phi(t, x_2)|_{x_2 = 0} \, d x_1 \\
        &= -L \phi_0(0) \prt_{x_2} \phi(t, x_2)|_{x_2 = 0}.
\end{align*}

The explicit solution to \refE{HeatShear} is
\begin{align*}
    \phi(t, z)
        &= \frac{1}{\sqrt{4 \pi \nu t}}
            \int_0^\iny \brac{e^{-\frac{(z - y)^2}{4 \nu t}}
                - e^{-\frac{(z + y)^2}{4 \nu t}}} \phi_0(y) \, dy
\end{align*}
(see, for instance, Section 3.1 of \cite{StraussPDE}). Thus,
\begingroup
\allowdisplaybreaks
\begin{align*}
    \prt_z \phi(t, z)|_{z = 0}
        &= -\frac{2}{4 \nu t \sqrt{4 \pi \nu t}}
            \int_0^\iny y \brac{e^{-\frac{y^2}{4 \nu t}}
                + e^{-\frac{y^2}{4 \nu t}}} \phi_0(y) \, dy \\
        &= -\frac{1}{\nu t \sqrt{4 \pi \nu t}}
            \int_0^\iny y e^{-\frac{y^2}{4 \nu t}} \phi_0(y) \, dy
                \\
        &= -\frac{1}{\nu t \sqrt{4 \pi \nu t}}
            \int_0^\iny (- 2 \nu t) \diff{}{y}
                e^{-\frac{y^2}{4 \nu t}} \phi_0(y) \, dy \\
        &= -\frac{1}{\sqrt{\pi \nu t}}
            \int_0^\iny \diff{}{y}
                e^{-\frac{y^2}{4 \nu t}} \, \phi_0(y) \, dy \\
        &= \frac{1}{\sqrt{\pi \nu t}}
            \int_0^\iny 
                e^{-\frac{y^2}{4 \nu t}} \phi_0'(y) \, dy
\end{align*}
\endgroup
so that
\begin{align*}
    \abs{\prt_{x_2} \phi(t, x_2)|_{x_2 = 0}}
        \le \frac{C}{\sqrt{\nu t}}.
\end{align*}
We conclude that
\begin{align*}
	\abs{\nu \int_0^T \int_\Gamma \omega \, \ol{u} \cdot \BoldTau}
	    \le C \sqrt{\nu} \int_0^T t^{-1/2} \, dt
	    = C \sqrt{\nu T}.
\end{align*}
The condition in \refE{BoundaryCondition2D} thus holds (as does \cref{e:nuL1Bound}). From \refR{ROC}, the rate of convergence is $C \nu^{\frac{1}{4}}$ (even for smoother initial data).

\Example{3} Consider Example 1a of radially symmetric vorticity in the unit disk, but without the assumption that $m$ given by \refE{m} vanishes. This example goes back at least to Matsui in \cite{Matsui1994}. The convergence also follows from the sufficiency of the Kato-like conditions established in \cite{TW1998}, as pointed out in \cite{W2001}. A more general convergence result in which the disk is allowed to impulsively rotate for all time appears in \cite{FLMT2008}. A simple argument to show that the vanishing viscosity limit holds is given in Theorem 6.1 \cite{K2006Disk}, though without a rate of convergence. Here we prove it with a rate of convergence by showing that the condition in \refE{BoundaryCondition2D} holds.

Because the nonlinear term disappears, the vorticity satisfies the heat equation, though with Dirichlet boundary conditions not on the vorticity but on the velocity:
\begin{align}\label{e:RadialHeat}
    \left\{
    \begin{array}{rl}
        \prt_t \omega = \nu \Delta \omega
            & \text{in } \Omega, \\
        u = 0
            & \text{on } \Gamma.
    \end{array}
    \right.
\end{align}
Unless $u_0 \in V$, however, $\omega \notin C([0, T]; L^2)$, so we cannot easily make sense of the initial condition this way.

An orthonormal basis of eigenfunctions satisfying these boundary conditions is
\begin{align*}
    u_k(r, \theta)
        &= \frac{J_1(j_{1k} r)}{\pi^{1/2}\abs{J_0(j_{1k})}}
            \wh{e}_\theta,
    \quad
    \omega_k(r, \theta)
        = \frac{j_{1k} J_0(j_{1k} r)}{\pi^{1/2}\abs{J_0(j_{1k})}},
\end{align*}
where $J_0$, $J_1$ are Bessel functions of the first kind and $j_{1k}$ is the $k$-th positive root of $J_1(x) = 0$.
(See \cite{K2006Disk} or \cite{LR2002}.) The $(u_k)$ are complete in $H$ and in $V$ and are normalized so that\footnote{This differs from the normalization in \cite{K2006Disk}, where $\norm{u_k}_H = j_{1k}^{-1}$, $\norm{\omega_k}_{L^2} = 1$.}
\begin{align*}
    \norm{u_k}_H = 1,
    \quad
    \norm{\omega_k}_{L^2} = j_{1k}.
\end{align*}

Assume that $u_0 \in H \cap H^1$. Then
\begin{align*}
    u_0 = \sum_{k = 1}^\iny a_k u_k,
    \quad
    \smallnorm{u_0}_H^2
        = \sum_{k = 1}^\iny  a_k^2
        < \iny.
\end{align*}
(But,
\begin{align*}
    \smallnorm{u_0}_V^2
        = \sum_{k = 1}^\iny  a_k^2 j_{1k}^2
        = \iny
\end{align*}
unless $u_0 \in V$.) We claim that
\begin{align*} 
    u(t)
        = \sum_{k = 1}^\iny a_k e^{- \nu j_{1k}^2 t} u_k
\end{align*}
provides a solution to the Navier-Stokes equations, ($NS$). To see this, first observe that $u \in C([0, T]; H)$, so $u(0) = u_0$ makes sense as an initial condition. Also, $u(t) \in V$ for all $t > 0$. Next observe that 
\begin{align*}
    \omega(t)
        := \omega(u(t))
        = \sum_{k = 1}^\iny a_k e^{- \nu j_{1k}^2 t} \omega_k
\end{align*}
for all $t > 0$, this sum converging in $H^n$ for all $n \ge 0$. Since each term satisfies \cref{e:RadialHeat} so does the sum. Taken together, this shows that $\omega$ satisfies \cref{e:RadialHeat} and thus $u$ solves ($NS$).

\Ignore{ 
\begin{align*}
    \sum_{k = 1}^\iny
            a_k^2 e^{- 2 \nu j_{1k}^2 t} \norm{\omega_k}_{L^2}^2
        =
        \sum_{k = 1}^\iny
            a_k^2 j_{1k} e^{- 2 \nu j_{1k}^2 t}
        < \iny
\end{align*}
for all $t > 0$
} 

The condition in \refE{BoundaryCondition2D} becomes
\begin{align*}
	\nu \int_0^T & \int_\Gamma \omega \, \ol{u} \cdot \BoldTau
	    = \nu \sum_{k = 1}^\iny \int_0^T \int_\Gamma
            a_k e^{- \nu j_{1k}^2 t} \omega_k
	        \, \ol{u} \cdot \BoldTau \, dt \\
	    &= \nu \sum_{k = 1}^\iny \int_0^T
            a_k e^{- \nu j_{1k}^2 t} \omega_k|_{r = 1}
	        \int_\Gamma \ol{u} \cdot \BoldTau \, dt\\
	    &= m \nu \sum_{k = 1}^\iny a_k
	            \frac{j_{1k} J_0(j_{1k})}
	            {\pi^{1/2}\abs{J_0(j_{1k})}}
	            \int_0^T
	            e^{- \nu j_{1k}^2 t} \, dt.
\end{align*}
In the final equality, we used
\begin{align*}
    \int_\Gamma \ol{u} \cdot \BoldTau
        = - \int_\Gamma \ol{u}^\perp \cdot \n
        = - \int_\Omega \dv \ol{u}^\perp
        = \int_\Omega \ol{\omega}
        = m.
\end{align*}
(Because vorticity is transported by the Eulerian flow, $m$ is constant in time.)

Then,
\begingroup
\allowdisplaybreaks
\begin{align*}
    &\abs{\nu \int_0^T \int_\Gamma \omega \, \ol{u} \cdot \BoldTau}
        \le \abs{m} \nu \sum_{k = 1}^\iny
            \frac{\abs{a_k}}{\pi^{1/2}} j_{1k}
	        \int_0^T
	            e^{- \nu j_{1k}^2 t} \, dt \\
	     &\qquad
	     = \abs{m} \nu \sum_{k = 1}^\iny
            \frac{\abs{a_k}}{\pi^{1/2}} j_{1k}
	        \frac{1 - e^{- \nu j_{1k}^2 T}}{\nu j_{1k}^2} \\
	     &\qquad
	     \le \frac{\abs{m}}{\pi^{\frac{1}{2}}}
	         \pr{\sum_{k = 1}^\iny a_k^2}^{\frac{1}{2}}
	        \pr{\sum_{k = 1}^\iny 
	            \frac{(1 - e^{- \nu j_{1k}^2 T})^2}{j_{1k}^2}}
	        ^{\frac{1}{2}} \\
	     &\qquad
	     = \frac{\abs{m}}{\pi^{\frac{1}{2}}}
	         \smallnorm{u_0}_H
	        \pr{\sum_{k = 1}^\iny 
	        \frac{(1 - e^{- \nu j_{1k}^2 T})^2}{j_{1k}^2}}
	        ^{\frac{1}{2}}.
\end{align*}
\endgroup
Classical bounds on the zeros of Bessel functions give $1 + k < j_{1k} \le \pi(\frac{1}{2} + k)$ (see, for instance, Lemma A.3 of \cite{K2006Disk}). Hence, with $M = (\nu T)^{-\al}$, $\al > 0$ to be determined, we have
\begingroup
\allowdisplaybreaks
\begin{align*}
    \sum_{k = 1}^\iny 
	        &\frac{(1 - e^{- \nu j_{1k}^2 T})^2}{j_{1k}^2}
        \le C  \sum_{k = 1}^\iny 
	        \frac{(1 - e^{- \nu k^2 T})^2}{k^2} \\
	    &\le (1 - e^{- \nu T})^2
	        + \int_{k = 1}^M 
	            \frac{(1 - e^{- \nu x^2 T})^2}{x^2} \, dx
	        + \int_{k = M + 1}^\iny 
	            \frac{(1 - e^{- \nu x^2 T})^2}{x^2} \, dx \\
	    &\le \nu^2 T^2
	        + \nu^2 T^2 \int_{k = 1}^M 
	            \frac{x^4}{x^2} \, dx
	        + \int_{k = M + 1}^\iny 
	            \frac{1}{x^2} \, dx \\
	    &\le \nu^2 T^2
	        + \nu^2 T^2 \frac{1}{3} \pr{M^3 - 1}
	        + \frac{1}{M} 
	    \le \nu^2 T^2
	        + \nu^2 T^2 M^3
	        + \frac{1}{M}  \\
	    &= \nu^2 T^2
	        + \nu^2 T^2 \nu^{-3 \al} T^{- 3 \al}
	        + (\nu T)^\al
	    =  \nu^2 T^2
	        + (\nu T)^{2 - 3 \al}
	        + (\nu T)^\al      
\end{align*}
\endgroup
as long as $\nu M^2 T \le 1$ (used in the third inequality); that is, as long as
\begin{align}\label{e:albetaReq}
    (\nu T)^{1 - 2 \al} \le 1.
\end{align}
Thus \cref{e:BoundaryCondition2D} holds (as does \cref{e:nuL1Bound}), so ($VV$) holds.

The rate of convergence in ($VV$) is optimized when $(\nu T)^{2 - 3 \al} = (\nu T)^\al$, which occurs when $\al = \frac{1}{2}$. The condition in \refE{albetaReq} is then satisfied with equality. \refR{ROC} then gives a rate of convergence in the vanishing viscosity limit of $C \nu^{\frac{1}{4}}$ (even for smoother initial data), except in the special case $m = 0$, which we note reduces to Example 1a.

\ReturnExample{1a} Let us apply our analysis of Example 3 to the special case of Example 1a, in which $u_0 \in V$. Now, on the boundary, 
\begin{align*}
    (\prt_t u + u \cdot \grad u + \grad p) \cdot \BoldTau
        = \nu \Delta u \cdot \BoldTau
        = \nu \Delta u^\perp \cdot (- \n)
        = - \nu \grad^\perp \omega \cdot \n.
\end{align*}
But $\grad p \equiv 0$ so the left-hand side vanishes. Hence, the vorticity satisfies homogeneous Neumann boundary conditions for positive time. (This is an instance of Lighthill's formula.) Since the nonlinear term vanishes, in fact, $\omega$ satisfies the heat equation, $\prt_t \omega = \nu \Delta \omega$ with homogeneous Neumann boundary conditions and hence $\omega \in C([0, T]; L^2(\Omega))$. 

Moreover, multiplying $\prt_t \omega = \nu \Delta \omega$ by $\omega$ and integrating gives
\begin{align*}
    \norm{\omega(t)}_{L^2}^2
        + 2 \nu \int_0^t \norm{\grad \omega(s)}_{L^2}^2 \, ds
            = \norm{\grad \omega_0}_{L^2}^2.
\end{align*}
We conclude that the $L^2$-norm of $\omega$, and so the $L^p$-norms for all $p \le 2$, are bounded in time uniformly in $\nu$. (In fact, this holds for all $p \in [1, \iny]$. This conclusion is not incompatible with \refT{VorticityNotBounded}, since $\ol{u} \equiv 0$ on $\Gamma$.)

This argument for bounding the $L^p$-norms of the vorticity fails for Example 3 because the vorticity is no longer continuous in $L^2$ down to time zero unless $u_0 \in V$. It is shown in \cite{FLMT2008} (and see \cite{GKLMN14}) that such control is nonetheless obtained for the $L^1$ norm.

\section{On a result of Bardos and Titi}\label{S:BardosTiti}

\noindent
Bardos and Titi in \cite{BardosTiti2013a, Bardos2014Private}, also starting from, essentially, \refE{VVArg}
make the observation that, in fact, for the vanishing viscosity limit to hold, it is necessary and sufficient that $\nu \omega$
(or, equivalently, $\nu [\prt_{\n} u]_{\BoldTau}$)
converge to zero on the boundary in a weak sense. In their result, the boundary is assumed to be $C^\iny$, but the initial velocity is assumed to only lie in $H$. Hence, the sufficiency condition does not follow immediately from \refE{VVArg}.

Their proof of sufficiency involves the use of dissipative solutions to the Euler equations. (The use of dissipative solutions for the Euler equations in a domain with boundaries was initiated in \cite{BardosGolsePaillard2012}. See also \cite{BSW2014}.) We present here the weaker version of their results in 2D that can be obtained without employing dissipative solutions. The simple and elegant proof of necessity is as in \cite{Bardos2014Private}, simplified further because of the higher regularity of our initial data. 

\begin{theorem}[Bardos and Titi \cite{BardosTiti2013a, Bardos2014Private}]\label{T:BardosTiti}
Working in 2D, assume that $\prt \Omega$ is $C^2$ and that $\ol{u} \in C^1([0, T; C^1(\Omega))$. Then for $u \to \ol{u}$ in $L^\iny(0, T; H)$ to hold it is necessary and sufficient that
\begin{align}\label{e:BardosNecCond}
    \nu \int_0^T \int_\Gamma \omega \, \varphi \to 0
        \text{ as } \nu \to 0
        \text{ for any } \varphi \in C^1([0, T] \times \Gamma).
\end{align}
\end{theorem}
\begin{proof}
Sufficiency of the condition follows immediately from setting $\varphi = (\ol{u} \cdot \BoldTau)|_\Gamma$ in \refT{BoundaryIffCondition}. 

To prove necessity, let $\varphi \in C^1([0, T] \times \Gamma)$. We will need a divergence-free vector field $v_\delta \in C^1([0, T]; H \cap C^\iny(\Omega))$ such that $v_\delta \cdot \BoldTau = \varphi$. Moreover, we require of $v_\delta$ that it satisfy the same bounds as the boundary layer corrector of Kato in \cite{Kato1983}; in particular,
\begin{align}\label{e:vBounds}
        \norm{\prt_t v_\delta}_{L^1([0, T]; L^2(\Omega))}
            \le C \delta^{1/2}, \qquad
        \norm{\grad v_\delta}_{L^\iny([0, T]; L^2(\Omega))}
            \le C \delta^{-1/2}.
\end{align}
This vector field can be constructed in several ways: we detail one such construction at the end of this proof.


The proof now proceeds very simply. We multiply the Navier-Stokes equations by $v_\delta$ and integrate over space and time to obtain
\begin{align}\label{e:BardosNec}
    \begin{split}
    \int_0^T (\prt_t &u, v_\delta)
        + \int_0^T (u \cdot \grad u, v_\delta)
        + \nu \int_0^T (\grad u, \grad v_\delta) \\
        &= \nu \int_0^T \int_\Gamma (\grad u \cdot \n)
            \cdot v_\delta
        = \nu \int_0^T \int_\Gamma \omega \, v_\delta
            \cdot \BoldTau
        = \nu \int_0^T \int_\Gamma \omega \, \varphi.
    \end{split}
\end{align}
Here, we used \refE{gradunomega} with $v_\delta$ in place of $\ol{u}$, and we note that no integrations by parts were involved.

Now, assuming that the vanishing viscosity limit holds, Kato shows in \cite{Kato1983} that setting $\delta = c \nu$---and using the bounds in \refE{vBounds}---each of the terms on the left hand side of \refE{BardosNec} vanishes as $\nu \to 0$. By necessity, then, so does the right hand side, giving the necessity of the condition in \refE{BardosNecCond}.

It remains to construct $v_\delta$.
To do so, we place coordinates on a tubular neighborhood, $\Sigma$, of $\Gamma$ as in the proof of \cref{L:Trace}. In $\Sigma$, define
\begin{align*}
    \psi(s, r) = - r \varphi(s).
\end{align*}
Write $\wh{r}$, $\wh{s}$ for the unit vectors in the directions of increasing $r$ and $s$. Then $\wh{r} \cdot \wh{s} = 0$ and $\wh{r} = - \n$ on $\Gamma$. Thus, on the boundary,
\begin{align*}
    \grad \psi(s, r)
        = -\varphi(s) \wh{r} -r \varphi'(s) \wh{s}.
\end{align*}
This gives
\begin{align*}
    \grad \psi(s) \cdot \n
        = -\varphi(s) \wh{r} \cdot \n
        = \varphi(s).
\end{align*}
It also gives $\grad \psi \in C^1([0, T]; C(\Sigma))$ so that $\psi \in \varphi \in C^1([0, T] \times \Sigma)$.

We now follow the procedure in \cite{Kato1983}. Let $\zeta: [0, \iny) \to [0, 1]$ be a smooth cutoff function with $\zeta \equiv 1$ on $[0, 1/2]$ and $\zeta \equiv 0$ on $[1, \iny]$. Define $\zeta_\delta(\cdot) = \zeta(\cdot/\delta)$ and
\begin{align*}
    v_\delta(x)
        = \grad^\perp (\zeta_\delta(\dist(x, \prt \Omega)) \psi(x)).
\end{align*}
Note that $v_\delta$ is supported in a boundary layer of width proportional to $\delta$. The bounds in \refE{vBounds} follow as shown in \cite{K2006Kato}.
\end{proof}

To establish the necessity of the stronger condition in \refT{BardosTiti}, we used (based on Bardos's \cite{Bardos2014Private}) a vector field supported in a boundary layer of width $c \nu$, as in \cite{Kato1983}. We used it, however, to extend to the whole domain an arbitrary cutoff function defined on the boundary, rather than to correct the Eulerian velocity as in \cite{Kato1983}.

\begin{remark}
    In this proof of \refT{BardosTiti}
    the time regularity in the test functions could be weakened
    slightly to assuming that
    $\prt_t \varphi \in L^1([0, T]; C(\Gamma))$,
    for this would still allow the first bound in
    \refE{vBounds} to be obtained.
\end{remark}
\begin{remark}
    Using the results of \cite{BardosTiti2013a, BSW2014} it is
    possible to change the condition in \refE{BardosNecCond} to
    apply to test functions $\varphi$ in
    $C^1([0, T]; C^\iny(\Gamma))$ (\cite{Bardos2014Private}).
    Moreover, this can be done
    without assuming time or spatial regularity of the
    solution to the Euler equations: only that the initial
    velocity lies in $H$.
\end{remark}



\ifbool{IncludeNavierBCSection}{
%
%
\section{Navier boundary conditions in 2D}\label{S:NavierBCs}

\refT{VorticityNotBounded} says that if the vanishing viscosity limit holds, then there cannot be a uniform (in $\nu$) bound on the $L^2$-norm of the vorticity. This is in stark contrast to the situation in the whole space, where such a bound holds, or for Navier boundary conditions in 2D, where such a bound holds for $L^p$, $p > 2$, as shown in \cite{FLP} and \cite{CMR}. For Navier boundary conditions in 2D, then, as long as the initial vorticity is in $L^p$ for $p > 2$ there will be a uniform bound on the $L^2$-norm of the vorticity, since the domain is bounded.


In fact, for Navier boundary conditions in 2D the classical vanishing viscosity limit ($VV$)
does hold, even for much weaker regularity on the initial velocity than that considered here (see \cite{KNavier}). The argument in the proof of \refT{VorticityNotBounded} then shows that
\begin{align}\label{e:VelocityGammaConvergence}
	u \to \ol{u}
		\text{ in } L^\iny([0, T]; L^2(\Gamma)).
\end{align}

We also have weak$^*$ convergence of the vorticity in $\Cal{M}(\ol{\Omega})$, as we show in \cref{T:VorticityConvergenceNavier}.

\begin{theorem}\label{T:VorticityConvergenceNavier}
Assume that the solutions to $(NS)$ are with Navier boundary conditions in 2D, and that  the initial vorticity $\omega_0 = \ol{\omega}_0$ is in $L^\iny$ (slightly weaker assumptions as in \cite{KNavier} can be made). Then all of the conditions in  \cref{T:VVEquiv} hold, but with the three conditions below replacing conditions $(C)$, $(E)$, and $(E_2)$, respectively:
\begin{align*}
	(C^N) & \qquad \grad u \to \grad \ol{u}
	        \quad
	        \weak^* \text{ in } L^\iny(0, T; \Cal{M}(\ol{\Omega})^{d \times d}), \\
	(E^N) & \qquad \omega \to \ol{\omega}
	        \quad
	        \weak^* \text{ in } L^\iny(0, T; \Cal{M}(\ol{\Omega})^{d \times d}), \\
	(E_2^N) & \qquad \omega \to \ol{\omega} 
	        \quad
	        \weak^* \text{ in } L^\iny(0, T; \Cal{M}(\ol{\Omega})). \\
\end{align*}
\end{theorem}
\begin{proof}
First observe that $(E^N)$ is just a reformulation of $(E_2^N)$ with vorticity viewed as a matrix. Also, it is sufficient to prove convergences in $(H^1(\Omega))^*$, using the same argument as in the proof of \cref{C:EquivConvMeasure}, since $\omega$ is bounded in all $L^p$ spaces, including $p = $.

It is shown in \cite{KNavier} that condition $(B)$ holds, from which $(A)$ and $(A')$ follow immediately. Condition $(D)$ is weaker than $(C^N)$ and condition $(F_2)$ is weaker than conditions $(E_2^N)$, so it remains only to show that $(C^N)$ and $(E_2^N)$ hold. We show this by modifying slightly the argument in the proof of  \cref{T:VVEquiv} given in \cite{K2008VVV}.

\medskip

	\noindent $\mathbf{(A') \implies (C^N)}$: Assume that ($A'$) holds and let $M$ be in
	$(H^1(\Omega))^{d \times d}$. Then
	\begin{align*}
		(\grad u, M)
			&= - (u, \dv M) + \int_\Gamma (M \cdot \mathbf{n}) \cdot u \\
			 &\to -(\ol{u}, \dv M) + \int_\Gamma (M \cdot \mathbf{n}) \cdot \ol{u}
					\text{ in } L^\iny([0, T]).
	\end{align*}
	The convergence follows from condition $(A')$ and \refE{VelocityGammaConvergence}.
	But,
		\begin{align*}
			-(\ol{u}, \dv M)
					= (\grad \ol{u}, M)
									- \int_\Gamma (M \cdot \mathbf{n}) \cdot \ol{u},
	\end{align*}
	giving ($C^N$).

	\medskip

	\noindent $\mathbf{(A') \implies (E_2^N)}$: Assume that ($A'$) holds and let $f$ be in
	$H^1(\Omega)$. Then
	\begin{align*}
		(\omega, f)
			&= - (\dv u^\perp, f) 
			= (u^\perp, \grad f) - \int_\Gamma (u^\perp \cdot \mathbf{n}) f \\
			&= - (u, \grad^\perp f) + \int_\Gamma (u \cdot \BoldTau) f \\
			&\to -(\ol{u}, \grad^\perp f) + \int_\Gamma (\ol{u} \cdot \BoldTau) f
					\text{ in } L^\iny([0, T])
	\end{align*}
	where $u^\perp = -\innp{u^2, u^1}$ and we used the identity $\omega(u) = - \dv u^\perp$
	and \refE{VelocityGammaConvergence}.
	But,
		\begin{align*}
			-(\ol{u}, &\grad^\perp f)
					= (\ol{u}^\perp, \grad f) 
					= - (\dv \ol{u}^\perp, f) 
							+ \int_\Gamma (\ol{u}^\perp \cdot \mathbf{n}) f \\
					&= - (\dv \ol{u}^\perp, f) 
							- \int_\Gamma (\ol{u} \cdot \BoldTau) f 
					= (\ol{\omega}, f) 
							- \int_\Gamma (\ol{u} \cdot \BoldTau) f,
	\end{align*}
	giving ($E_2^N$).
\end{proof}

\begin{remark}
If one could show that \refE{VelocityGammaConvergence} holds in dimension three then \refT{VorticityConvergenceNavier} would hold, with convergences in $(H^1(\Omega))^*$, in dimension three as well for initial velocities in $H^{5/2}(\Omega)$. This is because by \cite{IP2006} the vanishing viscosity limit holds for such initial velocities, and the argument in the proof of \refT{VorticityConvergenceNavier} would then carry over to three dimensions by making adaptations similar to those we made to the 2D arguments in \cite{K2008VVV}. Note that \refE{VelocityGammaConvergence} would follow, just as in 2D, from a uniform (in $\nu$) bound on the $L^p$-norm of the vorticity for some $p \ge 2$ if that could be shown to hold, though that seems unlikely.
\end{remark}
} 
{ 
}


\addtocontents{toc}{\protect\vspace{0.6em}}

\appendix
%
%
\section{Some Lemmas}

\noindent \refC{TraceCor}, which we used in the proof of \refT{VorticityNotBounded}, follows from \refL{Trace}.

\begin{lemma}[Trace lemma]\label{L:Trace}
	Let $p \in (1, \iny)$ and $q \in [1, \iny]$ be chosen
	arbitrarily, and let $q'$ be \Holder conjugate to $q$.
	There exists a constant $C = C(\Omega)$
    such that for all $f \in W^{1, p}(\Omega)$,
    \begin{align*}
        \norm{f}_{L^p(\Gamma)}
            \le C \norm{f}_{L^{(p - 1) q}(\Omega)}
            		^{1 - \frac{1}{p}}
                \norm{f}_{W^{1, q'}(\Omega)}
                	^{\frac{1}{p}}.
    \end{align*}
    If $f \in W^{1, p}(\Omega)$ has mean zero or $f \in W^{1, p}_0(\Omega)$ then
    \begin{align*}
        \norm{f}_{L^p(\Gamma)}
            \le C \norm{f}_{L^{(p - 1) q}(\Omega)}
            		^{1 - \frac{1}{p}}
                \norm{\grad f}_{L^{q'}(\Omega)}
                	^{\frac{1}{p}}.
    \end{align*}
\end{lemma}
\begin{proof}
We prove this for $f \in C^\iny(\Omega)$, the result following by the density of $C^\iny(\Omega)$ in $W^{1, p}(\Omega)$. We also prove it explicitly in two dimensions, though the proof extends easily to any dimension greater than two.

Let $\Sigma$ be a tubular neighborhood of $\Gamma$ of uniform width $\delta$, where $\delta$ is half of the maximum possible width. Place coordinates $(s, r)$ on $\Sigma$ where $s$ is arc length along $\Gamma$ and $r$ is the distance of a point in $\Sigma$ from $\Gamma$, with negative distances being inside of $\Omega$. Then $r$ ranges from $-\delta$ to $\delta$, with points $(s,0)$ lying on $\Gamma$. Also, because $\Sigma$ is only half the maximum possible width, $\abs{J}$ is bounded from below, where
\begin{align*}
    J = \det \frac{\prt(x, y)}{\prt (s, r)}
\end{align*}
is the Jacobian of the transformation from $(x, y)$ coordinates to $(s, r)$ coordinates.

Let $\varphi \in C^\iny(\Omega)$ equal 1 on $\Gamma$ and equal 0 on $\Omega \setminus \Sigma$. Then if $a$ is the arc length of $\Gamma$,
\begingroup
\allowdisplaybreaks
\begin{align*}
	\norm{f}_{L^p(\Gamma)}^p
		&= \int_0^a \int_{-\delta}^0 \pdx{}{r}
			\brac{(\varphi f)(s, r)}^p \, dr \, ds \\
		&\le \int_0^a \int_{-\delta}^0 \abs{\pdx{}{r}
			\brac{(\varphi f)(s, r)}^p} \, dr \, ds \\
		&\le \int_0^a \int_{-\delta}^0 \abs{\grad
			\brac{(\varphi f)(s, r)}^p} \, dr \, ds \\
		&= \pr{\inf_{\supp \varphi} \abs{J}}^{-1}
			\int_0^a \int_{-\delta}^0 \abs{\grad
			\brac{(\varphi f)(s, r)}^p}
			\inf_{\supp \varphi} \abs{J}
			\, dr \, ds \\
		&\le \pr{\inf_{\supp \varphi} \abs{J}}^{-1}
			\int_0^a \int_{-\delta}^0 \abs{\grad
			\brac{(\varphi f)(s, r)}^p}
			\abs{J}
			\, dr \, ds \\
		&= C
			\int_{\Sigma \cap \Omega} \abs{\grad
			\brac{(\varphi f)(x, y)}^p}
			\, dx \, dy \\
		&\le C
			\norm{\grad \brac{\varphi f}^p}_{L^1(\Omega)} \\
		&= C p
			\norm{(\varphi f)^{p - 1}
			\grad \brac{\varphi f}}_{L^1(\Omega)} \\
		&\le C p
			\norm{(\varphi f)^{p - 1}}_{L^q}
			\norm{\grad \brac{\varphi f}}_{L^{q'}(\Omega)} \\
		&= C p
			\brac{\int_{\Omega}{(\varphi f)^{{(p - 1)} q}}}
				^{\frac{1}{q}}
			\norm{\grad \brac{\varphi f}}_{L^{q'}(\Omega)} \\
		&= C p
			\norm{\varphi f}_{L^{(p - 1) q}(\Omega)}
				^{p - 1}
			\norm{\varphi \grad f + f \grad \varphi}
				_{L^{q'}(\Omega)} \\
		&\le C p
			\norm{f}_{L^{(p - 1) q}(\Omega)}
				^{p - 1}
			\norm{f}
				_{W^{1, q'}(\Omega)}.
\end{align*}
\endgroup
The first inequality then follows from raising both sides to the $\frac{1}{p}$ power and using $p^{1/p} \le e^{1/e}$. The second inequality follows from Poincare's inequality.
\end{proof}

\begin{remark}
    The trace inequality in \refL{Trace} is a folklore result,
    most commonly referenced in the special case where
    $p = q = q' = 2$. We proved it for completeness, since we
    could not find a proof (or even clear statement) in the literature.
    We also note that a simple, but incorrect, proof of it
    (for $p = q = q' = 2$) is
    to apply the \textit{invalid} trace inequality from
    $H^{\frac{1}{2}}(\Omega)$ to $L^2(\Gamma)$ then use
    Sobolev interpolation.
\end{remark}

Note that in \cref{L:Trace} it could be that $(p - 1) q \in (0, 1)$, though in our application of it in \cref{S:LpNormsBlowUp}, via \cref{C:TraceCor}, we have $(p - 1) q = 2$. Also, examining the last step in the proof, we see that for $p = 1$ the lemma reduces to $\norm{f}_{L^1(\Gamma)} \le C \norm{f}_{W^{1, q'}(\Omega)}$, which is not useful.

\begin{cor}\label{C:TraceCor}
	Let $p, q, q'$ be as in \cref{L:Trace}.
    For any $v \in H$,
    \begin{align*}
        \norm{v}_{L^p(\Gamma)}
            \le C \norm{v}_{L^{(p - 1) q}(\Omega)}
            		^{1 - \frac{1}{p}}
                \norm{\grad v}_{L^{q'}(\Omega)}
                	^{\frac{1}{p}}
    \end{align*}
    and for any $v \in V \cap H^2(\Omega)$,
    \begin{align*}
        \norm{\curl v}_{L^p(\Gamma)}
            \le C \norm{\curl v}_{L^{(p - 1) q}(\Omega)}
            		^{1 - \frac{1}{p}}
                \norm{\grad \curl v}_{L^{q'}(\Omega)}
                	^{\frac{1}{p}}.
    \end{align*}
\end{cor}
\begin{proof}
    If $v \in H$, then
    \begin{align*}
        \int_\Omega v^i
            & = \int_\Omega v \cdot \grad x_i
            = - \int_\Omega \dv v \, x_i
                + \int_\Gamma (v \cdot \n) x_i
            = 0.
    \end{align*}
    If $v \in V$ then
    \begin{align*}
        \int_\Omega \curl v
            &= - \int_\Omega \dv v^\perp
            = - \int_{\prt \Omega} v^\perp \cdot \n
            = 0.
    \end{align*}
    Thus, \refL{Trace} can be applied to $v_1, v_2$, and $\curl v$, giving the result.
\end{proof}

\section*{Acknowledgements} Work on this paper was supported in part by NSF Grants DMS-1212141 and DMS-1009545. The author benefited from several helpful conversations with Claude Bardos, Helena Nussenzveig Lopes, and Milton Lopes Filho that took place during his stay at the Instituto Nacional de Matem\'{a}tica Pura e Aplicada in Rio de Janeiro, Brazil in Spring 2014. He also acknowledges helpful conversations with Gung-Min Gie at the University of Louisville in Summer 2014. \cref{S:SomeConvergence} came about because of a question asked by Claude Bardos when the author gave a talk on \cite{K2008VVV} at Brown University in Fall 2009.

\IfarXivElse{

} 
{ 
\bibliography{Refs}
\bibliographystyle{plain}
}

\end{document}